\documentclass[12pt,a4paper]{article}
\usepackage{amsmath}
\usepackage[usenames,dvipsnames]{color}
\usepackage{graphics}
\usepackage{epsfig}
\usepackage{amsthm}
\usepackage{pdfpages}
\usepackage{amsmath,amsfonts,amssymb,latexsym}
\usepackage{color}
\usepackage{longtable}
\usepackage{colortbl}
\usepackage{hyperref}
\usepackage{makeidx}

\usepackage[whole]{bxcjkjatype}
\newtheorem{theorem}{Theorem}%[section]
\newtheorem{Thm}{Theorem}%[section]
\newtheorem{Rem2}[theorem]{Remark}%[section]
\newtheorem{Lem}[theorem]{Lemma}%[section]
	
\title{
{\Large
Explicit {\em a posteriori} local error estimation for FEM solutions \\
Hypercircle法による有限要素解の局所事後誤差評価について}
}
\author{中野 泰河(Taiga Nakano)\footnote{email: f18a055k@mail.cc.niigata-u.ac.jp} \and {劉 雪峰(Xuefeng Liu) \footnote{email: xfliu.math@gmail.com (corresponding author)}}}
\date{\normalsize 新潟大学大学院自然科学研究科\\(Graduate School of Science and Technology, Niigata University) 
\\ \quad \\ May 23, 2019} %% Could you please input no date data!!
\begin{document}
\maketitle

\begin{abstract}
In this paper, based on the Hypercircle method a local {\em a posteriori} error estimation for the finite element solution of Poisson's equation is proposed.
This method can also be effective for Poisson's equation which solution does not have $H^2$ regularity.
Numerical results for a square domain and an L-shaped domain are presented. \\

本論文ではHypercircle法を用いてPoisson方程式の有限要素解に対する局所事後誤差評価の手法を提案する．
当該手法の特徴として，解の$H^2$正則性によらずに関心のある領域における有限要素解の定量的な局所誤差評価を得ることができる．
また、2次元および3次元領域で定義されるPoisson方程式に対して数値実験を行って，提案手法の有効性を示した．
\end{abstract}

\section{はじめに}
  本研究では，Poisson方程式の境界値問題に対して，Hypercircle法による有限要素解の局所事後誤差評価手法を検討する.
  研究の背景として，半導体の抵抗率測定法を支配する偏微分方程式について，近似解の局所誤差評価が要求されることが挙げられる．
  具体的に，抵抗率測定に使用される4探針法(Fig.\ref{fig:four-probe}参照)では，測定試料の表面に4本の探針$A,B,C,D$を接触させ，探針$A,D$間に固定電流を流し，探針$B,C$間の電位差を測定することで抵抗率を測定する\cite{yamashita}．
  このとき，数値手法によって探針周囲での電位（Poisson方程式の境界値問題の解に対応する）の近似値を求め，近似値の局所誤差を評価することが要求される．
  \begin{figure}[h!]
   \begin{center}
   \includegraphics[width=2.8cm]{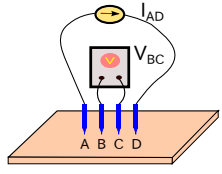} 
   \caption{Four probe method}
   \label{fig:four-probe}
   \end{center}
  \end{figure}

 有限要素解の大域的な事後誤差評価については既に多くの研究結果があり，特に近年，Hypercircle法を用いた有限要素解の誤差の具体的な値または誤差の上界評価(定量的な誤差評価)を提供する手法が報告される(\cite{liu-oishi,kikuchi3,qin-liu,Neittaanmaki,Braess}など)．
 特に，劉-大石の論文\cite{liu-oishi}では菊地の事後誤差評価\cite{kikuchi3}の拡張として事前誤差評価を提案したことがあり，この評価は本研究の提案手法に重要な役割を果たす．
 有限要素解の局所的な誤差評価については，誤差の収束オーダーなどの定性的な誤差評価(\cite{Demlow,Silvia}など)が考えられたが，定量的な有限要素解の局所誤差評価は少なかった．
% たとえば，文献\cite{Thomas}では，要素の間のfluxを利用して，要素毎の誤差評価式が提案された．だだし，当該評価式には$\mbox{div } \mathbf{p} + f = 0$を満たす$\mathbf{p}$を作成するために，各要素の境界における関数の処理には複雑な手間が必要となる．
 
 この論文では，Hypercircle法を利用することで，有限要素解の局所事後誤差評価における世界初の手法を提案する．提案手法の特徴として，L字型領域などの非凸な領域における方程式の厳密解が$H^2$正則性を持ってない場合でも，関心のある領域におけるシャープな誤差評価が可能である．
 抵抗率測定の場合， 提案手法によって非凸形状の測定試料に対しても，探針周囲での電位の誤差評価が可能となる．
 論文では，2次元領域と３次元領域の上で定義されるPoisson方程式の境界値問題に対して数値実験を行い，提案手法の有効性を検討した．
 
\section{準備}
\subsection{問題設定}
  この論文を通じて，特別な説明がない限り$\Omega$を$\mathbb{R}^2$の有界多角形領域とし，Sobolev空間$H^m(\Omega)$ $(m>0)$に関する標準的な記号\cite{miyashima}を用いる．
  $L^2(\Omega)$空間のノルムを$\| \cdot \|_{L^2(\Omega)}$または$\| \cdot \|_{\Omega}$とかき，
  記号$|\cdot|_{H^m(\Omega)},\| \cdot  \|_{H^m(\Omega)}$は，それぞれSobolev空間$H^m(\Omega)$のセミノルム，ノルムを表す．
  また記号$(\cdot,\cdot)$は$L^2(\Omega)$または$(L^2(\Omega))^2$の内積を表す．
  Sobolev空間$W^{1,\infty}(\Omega)$は，1階までの弱導関数が$\Omega$上で本質的に有界であるような関数空間である．
  ベクトル値関数の空間$H(\mbox{div};\Omega)$を以下で定義する．
  $$H(\mbox{div};\Omega):= \{ \mathbf{q} \in (L^2(\Omega))^2 ;~  \mbox{div } \mathbf{q} \in L^2(\Omega) \}$$

  次に以下のモデル問題を考える．
  \begin{eqnarray}
   -\Delta u = f  \mbox{ in } \Omega ,\quad  \displaystyle \frac{\partial u }{\partial \mathbf{n}} = g_N \mbox{ on } \Gamma_N,\quad u= g_D \mbox{ on } \Gamma_D
   \label{eq:model}
  \end{eqnarray}
  ここで$\Gamma_N,\Gamma_D$は，それぞれ互いに素な$\Omega$の境界$\partial \Omega$の部分集合であり，$\Gamma_N \cup \Gamma_D = \partial \Omega$を満たす．
  また$\frac{\partial }{\partial \mathbf{n}}$は$\partial \Omega$上の外向きの単位法線方向微分を表す．上の問題に対する弱形式は以下の式で与えられる．
  \begin{eqnarray}
   \mbox{Find } u \in V \mbox{ s.t. } (\nabla u,\nabla v) = (f,v) + (g_N,v)_{\Gamma_N} \quad \forall v \in V_0 \label{eq:model_weak}
  \end{eqnarray}
  このとき，試行関数の関数空間$V$と試験関数の関数空間$V_0$は，以下で定義される．
  \begin{equation}
  V := \{v \in H^1(\Omega); v =g_D \mbox{ on } \Gamma_D \},~V_0 := \{v \in H^1(\Omega); v =0 \mbox{ on } \Gamma_D \} 
  \end{equation}
  $\Gamma_D$が空集合となる場合，$V_0$の定義は以下のように変わる．
  \begin{equation}
  \label{eq:v0_new}
  V_0 := \{v \in H^1(\Omega); \int_\Omega v dx =0  \}
  \end{equation}
  また，
  $$(g_N,v)_{\Gamma_N} = \int_{\Gamma_N} g_N v ~ds .$$
  
  関数空間$V_0$に関して以下のPoincar\'e不等式が成り立つ．
   \begin{equation}
   \| \phi \|_{\Omega} \leq C_p \| \nabla \phi \|_{\Omega} \quad \forall v \in V_0 \label{eq:poincare_constant}
 \end{equation}
  
     \begin{Rem2}
    式(\ref{eq:poincare_constant})を満たす最適な定数$C_p$は固有値問題，
   $$\mbox{Find } (\lambda,\phi) \in \mathbb{R} \times V_0 \mbox{ s.t. } (\nabla \phi,\nabla v) = \lambda (\phi,v) \quad \forall v \in V_0$$
   の最小固有値$\lambda_1$に対して$C_p = 1/\sqrt{\lambda_1}$で与えられ，特に$\Omega$が単位正方形$(0,1)^2$かつ$\Gamma_D = \partial \Omega$の場合は$C_p=1/(\sqrt{2} \pi)$となる．
  一般的な領域と境界条件の設定の場合， $\lambda_1$の厳密な上下界は有限要素法によって得られる(\cite{VCB_Sec8,liu-oishi,liu-AMC}など)．
   文献\cite{liu-AMC}によると，Crouzeix-Raviart有限要素空間で定義される離散化固有値問題,
   $$\mbox{Find } (\lambda_h,\phi_h) \in \mathbb{R} \times V_0^h \mbox{ s.t. } (\nabla \phi_h ,\nabla v_h) = \lambda (\phi_h,v_h) \quad \forall v_h \in V_0^h$$
   の最小固有値$\lambda_{h,1}$を用いて，次の$\lambda_1$の評価式が成り立つ．
    $$ \lambda_1 \geq \frac{\lambda_{h,1}}{1+(0.1893h)^2\lambda_{h,1}} $$
  \end{Rem2}
  
  \subsection{有限要素法による近似解の構成}
  以降では議論を簡明にするためにモデル問題の境界条件に現れる$g_D,g_N$を,それぞれ境界上で区分的な1次関数，定数関数とする．
  一般の$g_D,g_N$の場合には，それぞれの関数近似の誤差評価が必要となる．
  
  領域$\Omega$の正則な三角形分割によって得られる三角形要素の族を$\mathcal{T}_h$とする．
  三角形要素$K \in \mathcal{T}_h$に対して$K$の最長辺の長さを$h_K$で表し，メッシュサイズ$h$を以下で定義する．
  $$h := \max_{K \in \mathcal{T}_h} h_K $$
  ここで関数空間$V,V_0$に対応する有限要素空間$V^h,V_0^h(\subset H^1(\Omega))$を以下で定義する．
  \begin{eqnarray*}
   V^h  &:=& \{v_h: \mbox{全領域$\Omega$で連続かつ各三角形要素上で区分的な1次関数}  \} \cap V \\
   V_0^h &:=& \{v_h: \mbox{全領域$\Omega$で連続かつ各三角形要素上で区分的な1次関数} \} \cap V_0 
  \end{eqnarray*}
  このとき(\ref{eq:model_weak})の適合有限要素法による定式化は以下で与えられる．
   \begin{eqnarray}
   \mbox{Find } u_h \in V^h \mbox{ s.t. } (\nabla u_h,\nabla v_h) = (f,v_h) + (g_N,v_h)_{\Gamma_N} \quad \forall v_h \in V_0^h \label{eq:CF}
  \end{eqnarray}
 さらに，以下の有限要素空間を準備する．
  \begin{itemize}
   \item[(a)] 区分定数関数空間$X^h$:
    $$X^h := \{v \in L^2(\Omega):v\mbox{は各三角形要素上で定数} \} $$
    $\Gamma_D$が空集合となる場合，$X^h$を以下のように定義する．
    $$X^h := \{v \in L^2(\Omega):v\mbox{は各三角形要素上で定数}, \int_\Omega v dx =0 ~ \} $$
   \item[(b)] $0$次Raviart-Thomas有限要素空間$W^h$:
    $$W^h := \{\mathbf{p}_h \in H(\mbox{div};\Omega):\mathbf{p}_h =(a_K+c_Kx,b_K+c_Ky) \mbox{ in } K \in \mathcal{T}_h \} $$
    ここで$a_K,b_K,c_K$は各三角形要素$K$上の定数である．
  \end{itemize}
  
%  以降では議論を簡明にするためにモデル問題の境界条件に現れる$g_D,g_N$を，それぞれ境界上で区分的な1次関数，定数関数とする．
% 一般の$g_D,g_N$の場合には，それぞれの誤差評価が必要である．
  上で述べた有限要素空間によって，問題(\ref{eq:model})の混合有限要素法による定式化は以下で与えられる．
  %ここで$\mathbf{p}_h$と$\mu_h$はそれぞれ$\nabla u$と$u$の近似となる．また，$\mathbf{p}_h \cdot n = g_N$を満たす$\mathbf{p}_h$を取るために，$g_N$を境界上で区分的な定数関数と仮定する．
%   
\begin{eqnarray}
 &&\quad\quad   \mbox{Find }  (\mathbf{p}_h,\mu_h) \in W^h \times X^h,~ \mathbf{p}_h \cdot n = g_N \mbox{ on } \Gamma_N \mbox{ s.t. }  \label{eq:MF}  \\ 
 &&\quad \quad  (\mathbf{p}_h,\mathbf{q}_h)  +  (\mbox{div } \mathbf{q}_h,\mu_h) + (\mbox{div } \mathbf{p}_h,\eta_h) = -(f,\eta_h)+ (g_D , (\mathbf{q}_h \cdot n) )_{\Gamma_D} ~~  \forall (\mathbf{q}_h ,\eta_h) \in  W^h \times    X^h  \nonumber
 \end{eqnarray}
 
 ここで，後述の誤差解析に使用される射影作用素$\pi_{h}:L^2(\Omega) \to X^h$を定義する．
 任意の$f \in L^2(\Omega)$に対して，$\pi_{h} f \in X^h$は
  \begin{equation}
   (f-\pi_{h} f,\eta_h) =0 \quad \forall \eta_h \in X^h
  \end{equation}
を満たし，さらに以下の誤差評価が成り立つ．
  \begin{equation}
   \| f- \pi_h f\|_{\Omega} \leq C_0h|f|_{H^1(\Omega)} \quad \forall f \in H^1(\Omega) \label{eq:inter_err}
  \end{equation}  
 ここで三角形分割に依存する定数$C_0$は$C_0:=\max_{K \in \mathcal{T}_h} C_0(K)/h$で定義され，明らかな上界を持つ．
 先行研究\cite{liu-kikuchi1,liu-kikuchi2,Laugesen}において，最適な$C_0(K)$の値は，Bessel関数$J_1$の正の最小の根$j_{1,1}\approx3.83171$によって$C_0(K):=h_K/j_{1,1}$と与えられることが報告されている．
   
 \section{有限要素解の大域的な事前誤差評価}
  この節では，以下の斉次混合境界値問題に対する有限要素解の大域的な事前誤差評価\cite{liu-oishi}(または\cite{VCB_Sec8} の8章，\cite{qin-liu}) を紹介する．
    \begin{eqnarray}
   -\Delta \phi = f  \mbox{ in } \Omega ,\quad  \displaystyle \frac{\partial \phi }{\partial \mathbf{n}} = 0 \mbox{ on } \Gamma_N,\quad \phi= 0 \mbox{ on } \Gamma_D
   \label{eq:homo_model}
  \end{eqnarray}
  %また，この有限要素解の大域的な事前誤差評価手法は，後に述べられる有限要素解の局所事後誤差評価を得るための重要な技巧を提供する．
  上の問題に対する弱形式と適合有限要素法による定式化は，以下で与えられる．
  \begin{eqnarray}
   \mbox{Find } \phi \in V_0 &\mbox{ s.t. }& (\nabla \phi,\nabla v) = (f,v)  \quad \forall v \in V_0 \label{eq:homo_model_weak}\\
   \mbox{Find } \phi_h \in V_0^h &\mbox{ s.t. }& (\nabla \phi_h,\nabla v_h) = (f,v_h)  \quad \forall v_h \in V_0^h \label{eq:homo_CF}
  \end{eqnarray}
  ここで，Galerkin射影作用素$P_h:V_0 \to V_0^h$を紹介する．
 任意の$v \in V_0$に対して，$P_h v$は以下の式を満たす．
  \begin{equation}
   (\nabla(v-P_h v),\nabla v_h) =0 \quad \forall v_h \in V_0^h \label{eq:galerkin}
  \end{equation}
  よって，$P_h \phi =\phi_h$となる．また，安定性に関する以下の評価式が成り立つ．
   \begin{equation}
   \| \nabla P_h \phi \|_{\Omega} \leq \| \nabla \phi \|_{\Omega} \leq C_p \| f \|_{\Omega} \label{eq:stablity-galerkin}
  \end{equation}
 ここで，$C_p$は式(\ref{eq:poincare_constant})に現れるPoincar\'e定数である．
  また，問題(\ref{eq:homo_model})の混合有限要素法による定式化は以下で与えられる．
  \begin{eqnarray}
 &&\quad\quad   \mbox{Find }  (\tilde{\mathbf{p}}_h,\sigma_h) \in W^h \times X^h,~ \tilde{\mathbf{p}}_h \cdot n = 0 \mbox{ on } \Gamma_N \mbox{ s.t. }  \label{eq:homo_MF}  \\ 
 &&\quad \quad  (\tilde{\mathbf{p}}_h,\mathbf{q}_h)  +  (\mbox{div } \mathbf{q}_h,\sigma_h) + (\mbox{div } \tilde{\mathbf{p}}_h,\tau_h) = -(f,\tau_h) \quad \forall (\mathbf{q}_h ,\tau_h) \in  W^h \times    X^h  \nonumber
 \end{eqnarray}
  
 有限要素解の大域的な事前誤差評価に重要な役割を果たす定数$\kappa_h$を以下で定義する．
 $$ \kappa_h := \max_{f_h \in X^h} ~~ \min_{ \substack{v_h \in V_0^h ,~ \tilde{\mathbf{p}}_h \in W^h, \\  \tilde{\mathbf{p}}_h \cdot n =0, ~ \text{div } \tilde{\mathbf{p}}_h + f_h =0} }  \frac{\| \nabla v_h -\tilde{\mathbf{p}}_h\|}{\|f_h\|} $$
 これらの定数$C_0,\kappa_h$を利用して，有限要素解の大域的な事前誤差評価を得ることができる．

\begin{Thm}[Prager-Syngeの定理\cite{synge}]
 関数$\phi$を式(\ref{eq:homo_model_weak})の厳密解とし，ベクトル値関数$\tilde{\mathbf{p}} \in H(\mbox{div};\Omega)$は以下を満たす．
 $$ \mbox{div }\tilde{\mathbf{p}}+f=0 ,~ \tilde{\mathbf{p}} \cdot n = 0 \mbox{ on }  \Gamma_N $$
 このとき，任意の$v \in V_0$に対して，Prager-Syngeの定理\cite{synge}に現れる以下のHypercircle式が成り立つ．
 \begin{equation} \|\nabla \phi - \nabla v \|_{\Omega}^2 + \| \nabla \phi- \tilde{\mathbf{p}}\|_{\Omega}^2 = \|\nabla v - \tilde{\mathbf{p}} \|_{\Omega}^2 \label{eq:Hypercircle} \end{equation}
 \label{Thm:PS}
 \end{Thm}
  
  \begin{Lem}[\cite{VCB_Sec8}の定理8.2, \cite{liu-oishi}の定理3.2]
 与えられた$f \in L^2(\Omega)$に対して，関数$\overline{\phi} \in V_0,~\overline{\phi}_h \in V_0^h$をそれぞれ以下の変分問題の解とする．
  \begin{eqnarray} 
  (\nabla \overline{\phi},\nabla v ) &= (\pi_h f,v)   \quad \forall v \in V_0 \label{eq:homo_aux}\\
  (\nabla \overline{\phi}_h,\nabla v_h ) &= (\pi_h f,v_h)   \quad \forall v_h \in V_0^h \label{eq:homo_aux_CF}
  \end{eqnarray}
  このとき，$\overline{\phi}_h$は$\overline{\phi}$の近似解であり，以下の誤差評価が得られる.
  \begin{equation} |\overline{\phi}-\overline{\phi}_h|_{H^1(\Omega)} \leq \kappa_h \|\pi_h f\|_{\Omega} \end{equation}
  \label{Thm:homo_aux}
 \end{Lem}
 \begin{Rem2}
  $f$の近似$f_h$に対して$\mbox{div } \tilde{\mathbf{p}}_h +f_h=0$を厳密に満たす$\tilde{\mathbf{p}}_h$を構成することが，提案手法の特徴の一つである．
  文献\cite{Thomas}は与えられる$f$に対して，要素の境界で$f$に対応する$\tilde{\mathbf{p}}$を構成している．
 \end{Rem2}
 
 \begin{Thm}[有限要素解の大域的な事前誤差評価\cite{liu-oishi}]
  与えられた$f \in L^2(\Omega)$に対して，$u$と$u_h$をそれぞれ変分問題(\ref{eq:homo_model_weak})と(\ref{eq:homo_CF})の解とする．
  このとき，以下の誤差評価が得られる．
  \begin{eqnarray} 
   && |\phi - \phi_h|_{H^1(\Omega)}   \leq   C(h) \|f\|_{\Omega} \label{eq:grobal-H1} \\
   && \|\phi - \phi_h \|_{\Omega}  \leq  C(h) |\phi - \phi_h|_{H^1(\Omega)} \label{eq:grobal-L2}
  \end{eqnarray}
  ここで，$C(h):=\sqrt{\kappa_h^2+(C_0h)^2}$である．
  \label{Thm:global_est}
 \end{Thm}
 \begin{Rem2}
方程式の解$u$が$H^2(\Omega)$に属するとき、
誤差評価式(\ref{eq:grobal-H1})と(\ref{eq:grobal-L2})の中の$C(h)$の代わりに、
Lagrange補間関数の誤差定数を使って値を与えることができる．
例えば、計算例\S5.2では、一様直角三角メッシュに対して、$C(h)$の代わりに$C_h=0.493h$を使うことも可能\cite{liu-kikuchi2,liu-oishi}．

 \end{Rem2}
 \begin{Rem2}
  定理\ref{Thm:global_est}において，  以下の事後誤差評価も得られる．
  $$
   |\phi-\phi_h|_{H^1(\Omega)} \leq C_0h \|f-\pi_hf \|_{\Omega} + \|\nabla \phi_h- \tilde{\mathbf{p}}_h \|_{\Omega}
  $$
 \end{Rem2}
 \begin{Rem2}
  定数$\kappa_h$の計算法:~
  与えられた$f_h \in X^h$に対して，変分問題(\ref{eq:homo_aux_CF})の解$\phi_h$を対応させる線形作用素$R_h:X^h \to V^h$と，
  変分問題(\ref{eq:homo_MF})の解$\tilde{\mathbf{p}}_h$を対応させる線形作用素$T_h:X^h \to W^h$を定義し，さらに$Q_h := \nabla \circ R_h - T_h$とする．
  このとき$\kappa_h$は以下のように計算されることが分かる．
  $$ \kappa_h = \max_{f_h \in X^h} \frac{\|Q_h f_h\|_{\Omega} }{\|f_h\|_{\Omega}} $$
  ここで$X^h,V^h,W^h$は有限次元ベクトル空間なので，上の式は行列の固有値問題に帰着される．
  より詳細な説明については先行研究\cite{liu-oishi,qin-liu}を参照する．
 \end{Rem2}
  \begin{Rem2}
  式(\ref{eq:model})に示した一般の混合境界値問題の場合，与えられるデータ$f,g_D,g_N$から解$u$または有限要素解$u_h$への対応はアフィン写像になる．
 \end{Rem2}
 
 \section{重み付きHypercircle式と有限要素解の局所事後誤差評価}
 この節では，Hypercircle式を用いた有限要素解の局所事後誤差評価を提案する．
 関心のある領域$S$は$\Omega$の部分領域として与えられ，$S$に対応する重み関数(またはCutoff関数)$\alpha$が導入される． 
 次に，重み関数$\alpha$を用いて重み付きHypercircle式を導入し，前節の技巧を応用して有限要素解の局所事後誤差評価を与える．
 \subsection{重み関数}
 関心のある領域$S$に対して，$S$を囲う帯状領域を$B_{S}:= \{ x \in \Omega \setminus \overline{S};~ \mbox{dist}(x,\partial S) < \varepsilon \}$とし(Fig.\ref{fig:alpha-graph}-(a),~Fig.\ref{fig:alpha-graph}-(b)参照)，さらに部分領域$\Omega' =S\cup B_{S}$を定義する．
 %\begin{enumerate}
%  関心のある領域$S$を囲う帯状領域を$B_{S}:= \{ x \in \Omega \setminus S;~ \mbox{dist}(x,\partial S) < \varepsilon \}$とする．
%  さらに，部分領域$\Omega' =S\cup B_{S}$を定義する．  
  次に，重み関数$\alpha \in W^{1,\infty}(\Omega)$を定義する．
  重み関数$\alpha$は$\overline{\Omega'}$を台(Support)とする区分的な多項式関数であり，以下の性質を持っている．
   \begin{equation}
     \alpha(x,y)  =
    \left\{ 
    \begin{array}{l}      1 \quad (x,y) \in S \\      0 \quad (x,y) \in (\Omega')^c     \end{array} 
    \right. , \quad  0 \leq \alpha(x,y) \leq 1 \quad \forall (x,y) \in \overline{\Omega}
   \end{equation}
 %\end{enumerate}
 
  \begin{figure}[h!]
    \begin{center}
    \includegraphics[width=3.5cm]{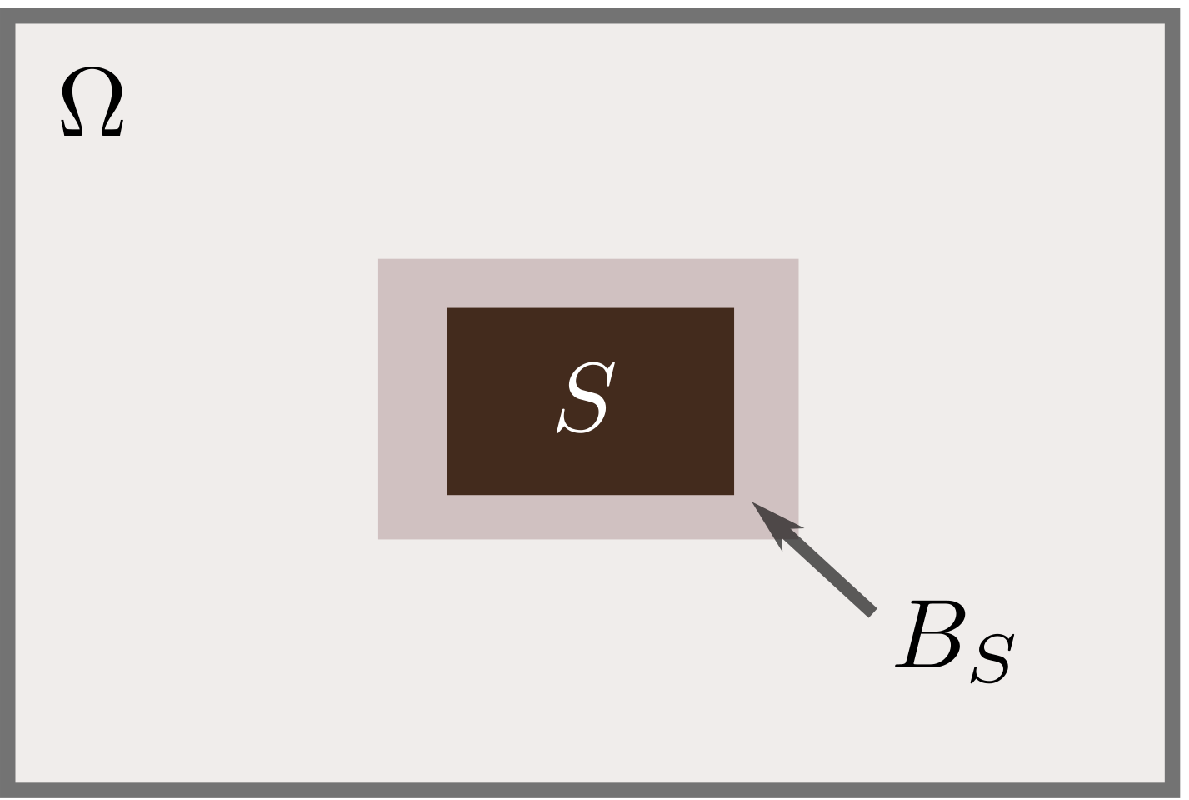} \quad
    \includegraphics[width=3.5cm]{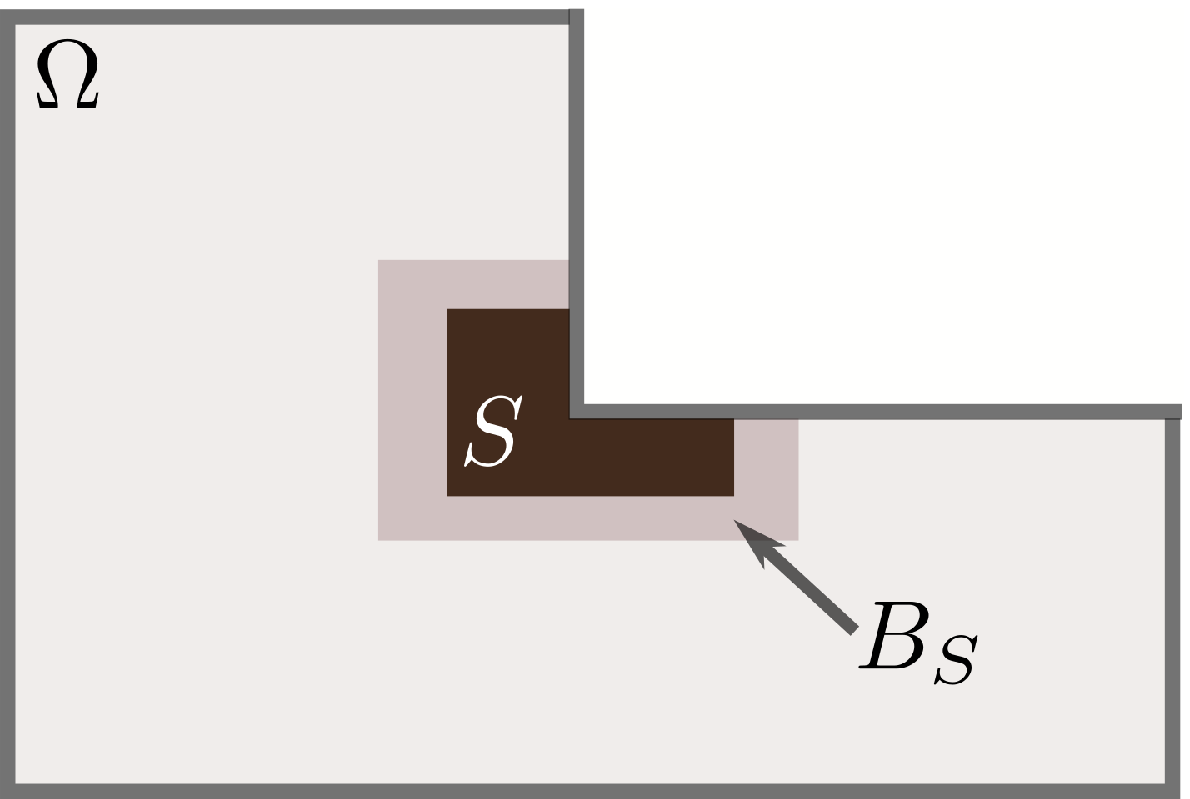} \quad 
    \includegraphics[width=4.2cm]{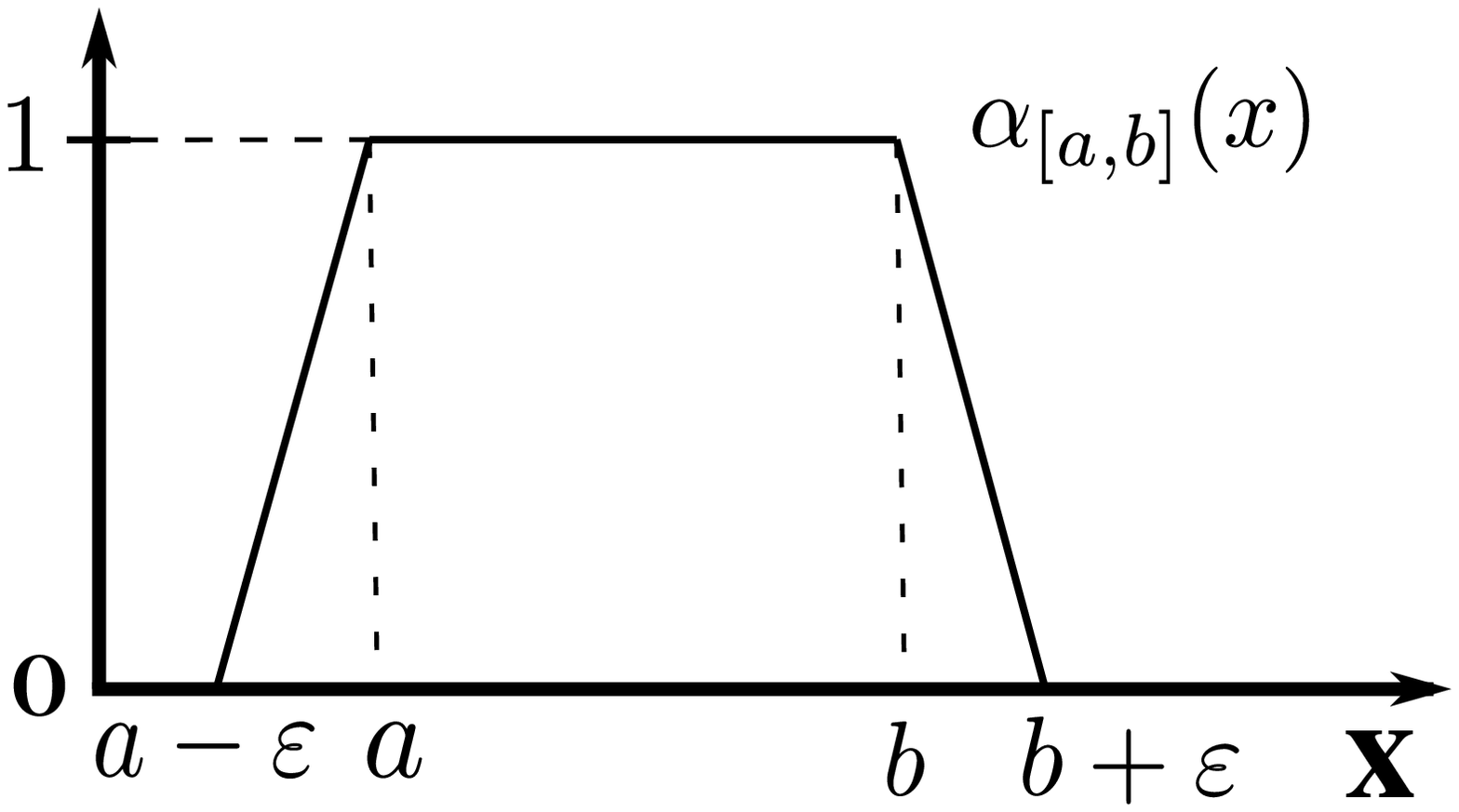}  \\   
    {(a) Square domain} \quad\quad \quad {(b) L-shpaed domain} \quad \quad \quad {(c) Graph of $\alpha_{(a,b)}(x)$}
    \caption{Definition of weight function $\alpha$}
    \label{fig:alpha-graph}
    \end{center}
\end{figure}

 特に，$S$が長方形の場合(Fig.\ref{fig:alpha-graph}-(a)参照)，$S$は2つの1次元開区間の直積によって表されるため，
 重み関数$\alpha$はそれぞれの開区間に対応する重み関数$\alpha_1(x)$と$\alpha_2(y)$の積によって構成できる．
 開区間$(a,b)$に対応する重み関数$\alpha_{(a,b)}$は以下のように構成できる．また，$\alpha_{(a,b)}$のグラフをFig.\ref{fig:alpha-graph}-(c)に示す．
% $$ 
%  \alpha_{[a,b]}(x) := \max \left \{ \min \left \{\left (1+\frac{b-a}{2 \varepsilon} \right ) \cdot \left ( 1- \frac{|x-\frac{a+b}{2}|}{\varepsilon +\frac{b-a}{2} } \right ),1\right \},0\right \}  
% $$
 \begin{eqnarray}
  \alpha_{(a,b)}(x)  =
  \left \{ 
   \begin{array}{ll}
   1+ (x-a)/\varepsilon  \quad & x \in (a-\varepsilon,a]  \\
   1 \quad & x \in(a,b)  \\
   1- (x-b)/\varepsilon  \quad & x \in [b,b+\varepsilon)  \\
   0 \quad & \mbox{otherwise}
  \end{array} 
  \right .
 \end{eqnarray}
 ただし，$\varepsilon$は帯状領域$B_{S}$の幅である．
 
 次に，重み関数$\alpha$によって，重み付き内積$(\cdot,\cdot)_{\alpha}$と重み付きノルム$\| \cdot \|_{\alpha}$を導入する．
 \begin{itemize} 
  \item[(a)] 重み付き内積$(\cdot,\cdot)_{\alpha}$: ~$f,g \in L^2(\Omega)$または$f,g \in (L^2(\Omega))^2$に対して，
  \begin{equation} (f,g)_{\alpha} := \int_{\Omega'}  \alpha f \cdot g ~dx. \end{equation}
  
  \item[(b)] 重み付きノルム$\| \cdot \|_{\alpha}:$ ~$f \in L^2(\Omega)$に対して， 
  \begin{equation} \| f\|_{\alpha} : = \sqrt{(f,f)_{\alpha}}= \sqrt{\int_{\Omega'} f^2 \alpha ~dx} \quad (=\| f\sqrt{\alpha} \|_{\Omega}). \end{equation}
  このとき，以下の不等式が成り立つことは容易に分かる．
 \begin{equation} \| f \|_S \leq \| f \|_{\alpha} \leq \| f \| _{\Omega'} \leq \| f \| _{\Omega} \label{eq:norm_relation} \end{equation}
 \end{itemize}

 \subsection{重み付きHypercircle式}
 この節では定理\ref{Thm:PS}の拡張として，重み付きノルムに関するHypercircle式を与える．
 \begin{Thm}
 関数$u$を式(\ref{eq:model_weak})の厳密解とし，ベクトル値関数$\mathbf{p} \in H(\mbox{div};\Omega)$は以下を満たす．
 $$ \mbox{div }\mathbf{p}+f=0 \mbox{ in } \Omega, \quad  \mathbf{p} \cdot n = g_N \mbox{ on }  \Gamma_N $$
 このとき，任意の$v \in V$に対して，以下の重み付きHypercircle式が成り立つ．
 \begin{equation} \|\nabla u - \nabla v \|_{\alpha}^2 + \| \nabla u- \mathbf{p}\|_{\alpha}^2 \leq \|\nabla v - \mathbf{p} \|_{\alpha}^2  + 2\|\nabla \alpha\|_{L^{\infty}(\Omega')} \|u-v \|_{\Omega'} \|\nabla u - \mathbf{p} \|_{\Omega'} \label{eq:alpha-Hypercircle} \end{equation}
 \label{Thm:alpha-PS}
 \end{Thm}
 \begin{proof}
 $\|\nabla v - \mathbf{p} \|_{\alpha}^2=\|(\nabla v - \nabla u )+ (\nabla u- \mathbf{p}) \|_{\alpha}^2$の展開式を考える．
 \begin{equation}  \|\nabla v - \mathbf{p} \|_{\alpha}^2  =   \|\nabla v - \nabla u \|_{\alpha}^2 + \| \nabla u- \mathbf{p}\|_{\alpha}^2  + 2 (\nabla v - \nabla u ,\nabla u- \mathbf{p})_{\alpha} \label{eq:alpha-expansion} \end{equation}
 上の展開式の交差項$(\nabla v - \nabla u ,\nabla u- \mathbf{p})_{\alpha} $について，$v-u$を$w$とおき，交差項を二つの項$(\nabla w,\nabla u)_{\alpha}$, $(\nabla w ,\mathbf{p})_{\alpha}$に分けて，それぞれの評価を与える．
 
 まず，$(\nabla w ,\nabla u)_{\alpha} $について，$\alpha w$に対して合成関数の微分公式$\alpha \nabla w =  -  w \nabla \alpha +  \nabla(\alpha w)$から以下の式が成り立つことが分かる．
 \begin{eqnarray}
    (\nabla w, \nabla u)_{\alpha}   = \int_{\Omega} \alpha \nabla w \cdot \nabla u ~dx & = &  - \int_{\Omega}  w (\nabla \alpha   \cdot \nabla u)  ~dx  + \int_{\Omega}  \nabla (\alpha w) \cdot \nabla u ~dx  \nonumber 
 \end{eqnarray}
 変分式(\ref{eq:model_weak})のテスト関数を$\alpha w$にすると，%と重み関数$\alpha$ の台が$\Omega'$であること，
$ (\nabla w,\nabla u)_{\alpha} $について以下の式が得られる．
 \begin{eqnarray}
    (\nabla w, \nabla u)_{\alpha}    = -\int_{\Omega }  w (\nabla \alpha   \cdot \nabla u)  ~dx +  \int_{\Omega} f (\alpha w)  ~dx  +\int_{\Gamma_N} g_N (\alpha w) ds  \label{eq:alpha-uw} 
 \end{eqnarray}
 次に，$(\nabla w ,\mathbf{p})_{\alpha}$について，Greenの公式を用いて以下の式が得られる．
 \begin{eqnarray}
 (\nabla w,\mathbf{p})_{\alpha} & = &\int_{\Omega} \nabla w \cdot (\alpha \mathbf{p})  ~dx = \int_{\Gamma_N} g_N (\alpha w) ds - \int_{\Omega} w \mbox{ div} (\alpha \mathbf{p}) ~dx \nonumber \\
 &=& \int_{\Gamma_N} g_N (\alpha w) ds- \int_{\Omega} (\alpha w)  \mbox{ div} \mathbf{p} ~dx - \int_{\Omega} w (\nabla \alpha \cdot \mathbf{p}) ~dx  \nonumber 
 \end{eqnarray}
 さらに，定理の仮定$ \mbox{div } \mathbf{p}+f=0$を利用することで以下の評価が得られる．
 \begin{eqnarray}
 (\nabla w,\mathbf{p})_{\alpha}  &=&  - \int_{\Omega} w (\nabla \alpha \cdot \mathbf{p}) ~dx + \int_{\Omega} (\alpha w) f ~dx +  \int_{\Gamma_N} g_N (\alpha w) ds   \label{eq:alpha-pw}
 \end{eqnarray}
 よって，式$(\ref{eq:alpha-uw}),~(\ref{eq:alpha-pw})$の差を取ることで，交差項に関して以下の等式が成り立つ．
 \begin{equation}
   (\nabla v - \nabla u ,\nabla u- \mathbf{p})_{\alpha} = - \int_{\Omega} (u-v) \nabla \alpha \cdot ( \nabla u- \mathbf{p} )~dx \label{eq:cross-term}
 \end{equation}
式 (\ref{eq:cross-term})の領域$\Omega$を$\Omega'$に入れ替えて，さらにH\"olderの不等式を用いることで，次の評価式が得られる．
 \begin{equation*}
   (\nabla v - \nabla u ,\nabla u- \mathbf{p})_{\alpha} \le 2 \| \nabla \alpha \|_{L^{\infty}(\Omega')} \| u- v\|_{\Omega'} \|\nabla u -\mathbf{p} \|_{\Omega'} 
 \end{equation*}
 最後に，式(\ref{eq:alpha-expansion})に対して上の結果を用いることで結論が得られる．
 \end{proof}
% \begin{Rem2}
%  $\displaystyle{ \| \nabla \alpha \|_{L^{\infty}(\Omega)}: = \esssup_{(x,y) \in \Omega}{\max\{|\alpha_x(x,y)|,|\alpha_y(x,y)|\}} }$
% \end{Rem2}
  \begin{Rem2}
  定理\ref{Thm:alpha-PS}の式(\ref{eq:alpha-uw}),~(\ref{eq:alpha-pw})において，関心のある領域$S$と$\Omega$が境界の一部を共有する場合でも，$\Gamma_N$上の積分は交差項の評価に影響を与えないことが分かる．
 \end{Rem2}

 \subsection{有限要素解の局所事後誤差評価}

 有限要素解の局所事後誤差評価を検討するために，与えられた$f \in L^2(\Omega)$に対して以下の補助問題を準備する．
  \begin{eqnarray} 
  (\nabla \overline{u},\nabla v ) &= (\pi_h f,v) + (g_N,v)_{\Gamma_N}  \quad \forall v \in V_0 \label{eq:aux}\\
  (\nabla \overline{u}_h,\nabla v_h ) &= (\pi_h f,v_h) + (g_N,v_h)_{\Gamma_N}  \quad \forall v_h \in V_0^h \label{eq:aux_CF}
  \end{eqnarray}
  
  \begin{Lem}
   $|u-\overline{u}|_{H^1(\Omega)} $に関して以下の評価が成り立つ．
   %$$\| \nabla (u-u_h) \|_S \leq C_0 h \| f - \pi_h f \|_{\Omega}+ \| \nabla (\overline{u} - u_h) \|_{\alpha}$$
   \begin{equation} |u-\overline{u}|_{H^1(\Omega)} \leq C_0h \|f-\pi_hf \|_{\Omega}  \label{eq:second} \end{equation}
   \label{Lem:aux}
  \end{Lem}
  \begin{proof}
%  式(\ref{eq:norm_relation})を用いて，$\| \nabla (u-u_h) \|_S$に関する以下の展開式が得られる．
%  \begin{equation} \| \nabla (u-u_h) \|_S \leq \| \nabla (u-\overline{u}) \|_S + \| \nabla (\overline{u} - u_h) \|_S \leq \| \nabla (u-\overline{u}) \|_{\Omega}+ \| \nabla (\overline{u} - u_h) \|_{\alpha} \label{eq:expansion2} \end{equation}
$\| \nabla (u-\overline{u}) \|_{\Omega}$について，$u$と$\overline{u}$の定義から以下の等式が成り立つ．
  $$(\nabla(u-\overline{u}),\nabla v) = (f-\pi_h f,v) = (f-\pi_h f,v-\pi_h v) \quad \forall v \in V_0$$
  射影作用素$\pi_h$に関する誤差評価式(\ref{eq:inter_err})から以下の評価式を得る．
  $$(\nabla(u-\overline{u}),\nabla v) \leq \|f-\pi_hf \|_{\Omega} \cdot C_0h |v |_{H^1(\Omega)} \quad \forall v \in V_0$$
  試験関数$v$を$u-\overline{u}$とすることで，以下の評価が得られる．
  $$ |u-\overline{u}|_{H^1(\Omega)} \leq C_0h \|f-\pi_hf \|_{\Omega} $$
  \end{proof}
  
  \begin{Lem}
  与えられた$f \in L^2(\Omega)$に対して，変分問題(\ref{eq:model_weak})と(\ref{eq:CF})の解)の解をそれぞれ$u$と$u_h$とし，
  変分問題(\ref{eq:aux})と(\ref{eq:aux_CF})の解をそれぞれ$\overline{u},\overline{u}_h$とする.
  さらに，$\mathbf{p}_h \in H(\mbox{div};\Omega)$は以下の条件を満たすベクトル値関数とする．
  $$ \mbox{div }\mathbf{p}_h+\pi_h f=0,~ \mathbf{p}_h \cdot n = g_N \mbox{ on }  \Gamma_N $$
   このとき$\|\nabla (\overline{u} - u_h)\|_{\alpha}$に関して，以下の評価式が得られる．
  $$
    \|\nabla (\overline{u} - u_h)\|_{\alpha}^2  \leq \|\nabla u_h - \mathbf{p}_h\|_\alpha^2 
    +2 \|\nabla \alpha\|_{L^{\infty}(\Omega')} \|\overline{u}-u_h\|_{\Omega} ~ \|\nabla \overline{u} - \mathbf{p}_h\|_{\Omega } 
  $$
  \label{Lmm:aux-hypercircle}
  \end{Lem}
  \begin{proof}
   定理\ref{Thm:alpha-PS}において$f = \pi_h f $とすると，以下の重み付きHypercircle式が得られる．
  $$
    \|\nabla (\overline{u} - u_h)\|_{\alpha}^2 + \|\nabla \overline{u} - \mathbf{p}_h \|_\alpha^2  \leq \|\nabla u_h - \mathbf{p}_h\|_\alpha^2 
    +2 \|\nabla \alpha\|_{L^{\infty}(\Omega')} \|\overline{u}-u_h\|_{\Omega'} ~ \|\nabla \overline{u} - \mathbf{p}_h\|_{\Omega' } 
  $$
  式(\ref{eq:norm_relation})の結果により，以下の評価式が得られる．
  \begin{equation}
   \| \nabla (\overline{u} - u_h)\|_{\alpha}^2  \leq \|\nabla u_h - \mathbf{p}_h\|_\alpha^2 
    +2 \|\nabla \alpha\|_{L^{\infty}(\Omega')} \|\overline{u}-u_h\|_{\Omega} ~ \|\nabla \overline{u} - \mathbf{p}_h\|_{\Omega} \label{eq:u-overu-alpha}
  \end{equation}
  \end{proof}

 \begin{Thm}[有限要素解の局所事後誤差評価]
  与えられた$f \in L^2(\Omega)$に対して，$u$と$u_h$を，それぞれ変分問題(\ref{eq:model_weak})と(\ref{eq:CF})の解とし，
  $\mathbf{p}_h \in H(\mbox{div};\Omega)$は以下の条件を満たすベクトル値関数とする．
  $$ \mbox{div }\mathbf{p}_h+\pi_h f=0,~ \mathbf{p}_h \cdot n = g_N \mbox{ on }  \Gamma_N $$
  このとき，以下の誤差評価が得られる．
  \begin{equation} \|\nabla u - \nabla u_h \|_{S}   \leq C_0h \|f -  \pi_h f\|_{\Omega} + \sqrt{(Err^{(1)})^2+(Err^{(2)})^2+(Err^{(3)})^2} \label{eq:local_est} \end{equation}
  ただし，  
  \begin{eqnarray*}
   (Err^{(1)})^2 &:= &2 C_p  C_0h  \|\nabla \alpha\|_{L^{\infty}(\Omega')} \cdot   \|f-\pi_h f \|_{\Omega} \cdot \|\nabla u_h - \mathbf{p}_h\|_{\Omega} , \\ 
   (Err^{(2)})^2 &:= & 2 C(h) \|\nabla \alpha\|_{L^{\infty}(\Omega')} \cdot   \|\nabla u_h - \mathbf{p}_h\|_{\Omega}^2 ,\\
   (Err^{(3)})^2 &:= & \| \nabla u_h -\mathbf{p}_h \|_{\alpha}^2.
  \end{eqnarray*}
  ここで，$C_p$は式(\ref{eq:poincare_constant})に現れるPoincar\'e定数である．
  \label{Thm:local_est}
 \end{Thm}
 \begin{proof}
  %補題\ref{Lem:aux}にから$\| \nabla (\overline{u} - u_h) \|_{\alpha}$に着目し，
  まず，式(\ref{eq:norm_relation})を用いて，$\| \nabla (u-u_h) \|_S$に関する以下の展開式が得られる．
  \begin{equation} \| \nabla (u-u_h) \|_S \leq \| \nabla (u-\overline{u}) \|_S + \| \nabla (\overline{u} - u_h) \|_S \leq \| \nabla (u-\overline{u}) \|_{\Omega}+ \| \nabla (\overline{u} - u_h) \|_{\alpha} \label{eq:expansion2} \end{equation}
   式(\ref{eq:expansion2}),~(\ref{eq:second})から，以下の評価が得られる．
  \begin{equation} \| \nabla (u-u_h) \|_S \leq C_0 h \| f - \pi_h f \|_{\Omega}+ \| \nabla (\overline{u} - u_h) \|_{\alpha} \label{eq:local1} \end{equation}
  定理\ref{Thm:alpha-PS}において$f = \pi_h f , v = u_h$とすると，重み付きHypercircle式(\ref{eq:alpha-Hypercircle})によって，以下の評価式が得られる．
  \begin{equation}
   \| \nabla (\overline{u} - u_h)\|_{\alpha}^2  \leq \|\nabla u_h - \mathbf{p}_h\|_\alpha^2 
    +2 \|\nabla \alpha\|_{L^{\infty}(\Omega')} \|\overline{u}-u_h\|_{\Omega} ~ \|\nabla \overline{u} - \mathbf{p}_h\|_{\Omega} \label{eq:u-overu-alpha}
  \end{equation}

  式(\ref{eq:local1}),(\ref{eq:u-overu-alpha})から，$\| \nabla (u-u_h) \|_S$に関して以下の評価式を得る．
  \begin{equation}
   \| \nabla (u-u_h) \|_S \leq C_0 h \| f - \pi_h f \|_{\Omega} +\sqrt { \|\nabla u_h - \mathbf{p}_h\|_\alpha^2 +2 \|\nabla \alpha\|_{L^{\infty}(\Omega')} \|\overline{u}-u_h\|_{\Omega} ~ \|\nabla \overline{u} - \mathbf{p}_h\|_{\Omega} }
   \label{eq:origin_main}
  \end{equation}
 
  次に，$\| \nabla (u-u_h) \|_S$の評価式(\ref{eq:origin_main})にある$\|\overline{u}-u_h\|_{\Omega},~\|\nabla \overline{u} - \mathbf{p}_h\|_{\Omega} $を検討する．
  \begin{itemize}
   \item[(a)] $\|\nabla \overline{u} - \mathbf{p}_h\|_{\Omega} $の評価: ~$\overline{u} \in V$と$ \mbox{div }\mathbf{p}_h+\pi_hf=0 ,~\mathbf{p}_h \cdot n = g_N \mbox{ on }  \Gamma_N $を満たすベクトル値関数$\mathbf{p_h} \in W^h$に対して，以下のHypercircle式が得られる．
  $$\|\nabla \overline{u} - \nabla v_h \|_{\Omega}^2 + \| \nabla \overline{u}- \mathbf{p}_h\|_{\Omega}^2 = \|\nabla v_h - \mathbf{p}_h \|_{\Omega}^2 \quad  \forall v_h \in V^h $$
  よって，以下の評価式が得られる．
  \begin{eqnarray}
   \|\nabla \overline{u} - \nabla u_h \|_{\Omega} & \leq & \|\nabla u_h - \mathbf{p}_h \|_{\Omega} \label{eq:overu-v}\\
   \| \nabla \overline{u}-\mathbf{p}_h \|_{\Omega} & \leq & \|\nabla u_h - \mathbf{p_h} \|_{\Omega} \label{eq:up-est}
  \end{eqnarray}
  
    \item[(b)] $\|\overline{u}-u_h\|_{\Omega}$の評価: ~以下の双対問題を定義する．
    $$\mbox{Find } \phi \in V_0 \mbox{ s.t. } (\nabla \phi,\nabla v) = (\overline{u}-u_h,v) \quad \forall v \in V_0 $$
    双対問題の解$\phi$に対して，$\phi_h:=P_h \phi$とすると（$P_h$は式(\ref{eq:galerkin})で定義したGalerkin射影である），$\|\overline{u}-u_h\|_{\Omega}$に関して以下の式が得られる．
    ここで$u_h$は$\overline{u}$のGalerkin射影ではないことに注意する．
    \begin{eqnarray*}
     \| \overline{u}-u_h \|_{\Omega}^2 &= & (\overline{u}-u_h,\overline{u}-u_h) 
                                         = (\nabla \phi,\nabla(\overline{u}-u_h)) \\
                                         & = & (\nabla \phi_h,\nabla(\overline{u}-u_h)) + (\nabla (\phi-\phi_h),\nabla(\overline{u}-u_h))  \\
                                         & \leq & (\nabla \phi_h,\nabla(\overline{u}-u_h)) + \|\nabla (\phi-\phi_h) \|_{\Omega} \cdot \|\nabla(\overline{u}-u_h) \|_{\Omega} 
    \end{eqnarray*}
   さらに式(\ref{eq:overu-v})を用いて，以下の展開式が得られる．
    $$\| \overline{u}-u_h \|_{\Omega}^2 \leq (\nabla \phi_h,\nabla(\overline{u}-u_h)) + \|\nabla (\phi-\phi_h) \|_{\Omega} \cdot \|\nabla u_h - \mathbf{p}_h \|_{\Omega} $$
    以降では展開式の$(\nabla \phi_h,\nabla(\overline{u}-u_h))$と$\|\nabla (\phi-\phi_h) \|_{\Omega}$について，それぞれの評価を与える．
    まず，展開式の$(\nabla \phi_h,\nabla(\overline{u}-u_h))$について，射影作用素$\pi_h$に関する誤差評価式(\ref{eq:inter_err})から以下の評価式を得る．
   \begin{eqnarray*} 
    (\nabla \phi_h,\nabla(\overline{u}-u_h)) & = & (\phi_h,\pi_h f -f ) = (\phi_h - \pi_h \phi_h ,\pi_h f -f ) \notag \\ 
    & \leq & C_0h \|\nabla \phi_h \|_{\Omega} \|f-\pi_h f \|_{\Omega} \\
    & \leq  & C_p C_0h \| \overline{u} - u_h \|_{\Omega} \|f-\pi_h f \|_{\Omega}
    %\label{eq:errterm} 
   \end{eqnarray*}
   上の不等式の最後に，Galerkin射影作用素に関する安定性の評価式(\ref{eq:stablity-galerkin})を用いた．
%   \begin{equation}
%    \| \nabla \phi_h \|_{\Omega} \leq \| \nabla \phi \|_{\Omega} 
%											 \leq C_p \|\overline{u} - u_h \|_{\Omega}
%   \label{eq:stablity}
%   \end{equation}
%   式(\ref{eq:errterm})に対して，式(\ref{eq:stablity})を用いることで$(\nabla \phi_h,\nabla(\overline{u}-u_h))$に関する以下の評価式が得られる．
%   \begin{eqnarray*} 
%    (\nabla \phi_h,\nabla(\overline{u}-u_h))  \leq C_p C_0 h  \|\overline{u} - u_h \|_{\Omega} \|f-\pi_h f \|_{\Omega} 
%   \end{eqnarray*}
   次に，展開式の$ \|\nabla (\phi-\phi_h) \|_{\Omega}$ついて，定理\ref{Thm:global_est}の事前誤差評価式(\ref{eq:grobal-H1})から以下の評価式を得る．
   $$\|\nabla (\phi-\phi_h) \|_{\Omega} \leq C(h)  \|\overline{u}-u_h \|_{\Omega} $$
   
   よって，$\| \overline{u}-u_h \|_{\Omega}$に関して以下の評価式が得られる．
   \begin{eqnarray} \| \overline{u}-u_h \|_{\Omega} 
    & \leq  C_p C_0 h\|f-\pi_h f \|_{\Omega} + C(h) \|\nabla u_h - \mathbf{p_h} \| _{\Omega} \label{eq:u-overu-L2est} 
   \end{eqnarray}
   
  \end{itemize}
  目標としている式(\ref{eq:origin_main})の根号の中の評価は式(\ref{eq:u-overu-alpha})に対して，評価式(\ref{eq:up-est}),~(\ref{eq:u-overu-L2est})を用いることで得られる．
  \begin{eqnarray*}
   &\quad&{  \|\nabla u_h - \mathbf{p}_h\|_\alpha^2 +2 \|\nabla \alpha\|_{L^{\infty}(\Omega')} \|\overline{u}-u_h\|_{\Omega} ~ \|\nabla \overline{u} - \mathbf{p}_h\|_{\Omega} } \\
   && \leq { 2 \|\nabla \alpha\|_{L^{\infty}(\Omega')} \left ( C_p C_0 h\|f-\pi_h f \|_{\Omega} + C(h) \|\nabla u_h - \mathbf{p_h} \| _{\Omega}  \right )  \|\nabla u_h - \mathbf{p_h} \| _{\Omega}+ \|\nabla u_h - \mathbf{p}_h\|_\alpha^2} \\
   && = {(Err^{(1)})^2 +(Err^{(2)})^2 + (Err^{(3)})^2}
  \end{eqnarray*}
  以上により定理の結論を得る．
 \end{proof}
 
 \begin{Rem2}
   定理\ref{Thm:local_est}の中に現れる$Err^{(1)},Err^{(2)}$および$Err^{(3)}$について，$Err^{(1)},Err^{(2)}$は有限要素解の大域的な誤差に関わる項で，$Err^{(3)}$は有限要素解の局所的な誤差に関わる主要項である．
   $Err^{(1)},Err^{(2)}$のオーダーは$O(h^{1.5})$，$Err^{(3)}$のオーダーは$O(h^{1})$となることが期待される．特に，メッシュサイズ$h$が十分小さいとき主要項$Err^{(3)}$が支配的になることが期待される．
   \label{rem:order}
  \end{Rem2}
   
  \section{数値実験}
 \subsection{準備}
 5.2,~5.3節では，2次元領域で定義された，以下の斉次Dirichlet問題に対して数値実験を行う．
 \begin{eqnarray}
 \left \{
  \begin{array}{l}
  - \Delta u = 2 \pi \sin(\pi x) \sin(\pi y) \mbox{ in } \Omega \\
  u = 0 \mbox{ on } \partial \Omega 
  \end{array}
  \right . 
  \label{eq:experiment}
 \end{eqnarray}
  5.4節では，抵抗率測定の問題に応じて，3次元領域で定義された混合境界値問題について検討する．
  
 5.2,~5.3節では，最初に帯状領域$B_{S}$の幅を決定する．$B_{S}$の幅を大きく取ると，誤差評価式(\ref{eq:local_est})に現れる主要項$Err^{(3)}$が大きなる一方で，$B_{S}$の幅を小さく取ると，領域全体の誤差評価に関わる項${Err^{(1)}},{Err^{(2)}}$に含まれる$\| \nabla \alpha\|_{L^{\infty}(\Omega)}$が大きくなる．
 そこで，$B_{S}$の幅を$0.2 \sim 0.375$まで$0.05$ずつ取って，$B_{S}$の幅の影響を調べる．次に，適切な$B_{S}$の幅を取って，メッシュサイズ$h$と誤差評価式(\ref{eq:local_est})の関係を調べる．
  
   5節を通して，誤差評価式(\ref{eq:local_est})に現れる記号${Err^{(1)}},{Err^{(2)}},{Err^{(3)}}$を引き続き使用し，誤差評価式(\ref{eq:local_est})を計算することで得られた値を$\overline{E}_{L}$で表す．  
  すなわち，
  $$\overline{E}_{L} := C_0h \|f -  \pi_h f\|_{\Omega} + \sqrt{(Err^{(1)})^2+(Err^{(2)})^2+(Err^{(1)})^2}  .$$
  さらに以下の記号を導入する．
  $$ 
   E_{L} := \| \nabla u - \nabla u_h  \|_{S},~ \overline{E}_{G} := \| \nabla u_h - \mathbf{p}_h  \|_{\Omega} + C_0h \|f -\pi_h f \|_{\Omega}
  $$
  
\subsection{正方形領域}
  正方形領域の場合，問題(\ref{eq:experiment})の厳密解$u = \sin(\pi x)\sin(\pi y)$を利用して$\overline{E}_{L}$と$ E_{L}$の比較を行うことができる．
  数値実験の条件を以下に示す．
 \begin{enumerate}
  \item[(a)] 領域$\Omega  := (0,1)^2$， 部分領域$S: = (0.375,0.625)^2$．
  \item[(b)] 一様メッシュのサイズ$\displaystyle h:=\frac{1}{8},~\frac{1}{16},~\frac{1}{32},~\frac{1}{64},~\frac{1}{128}$．
 \end{enumerate} 
 
 \noindent \textbf{帯状領域$B_{S}$の幅の選択} 　メッシュサイズ$h= 1/64,1/128$の場合に帯状領域$B_{S}$の幅と$\overline{E}_{L}$の関係をFig.\ref{fig:RD-tol}-(a),(b)に示す．
 この図から最適な$B_{S}$の幅は約$0.3$であり，この付近での幅の$1 \%$の変化に対して，$\overline{E}_{L}$の変化が約$1 \%$となっていることが分かる．
 したがって，最適な$B_{S}$の幅の周りでは幅の変化が与える誤差評価への影響は小さいと考えられる．  
   \begin{figure}[h!]
     	\begin{center}
	\includegraphics[width=3cm,angle=270]{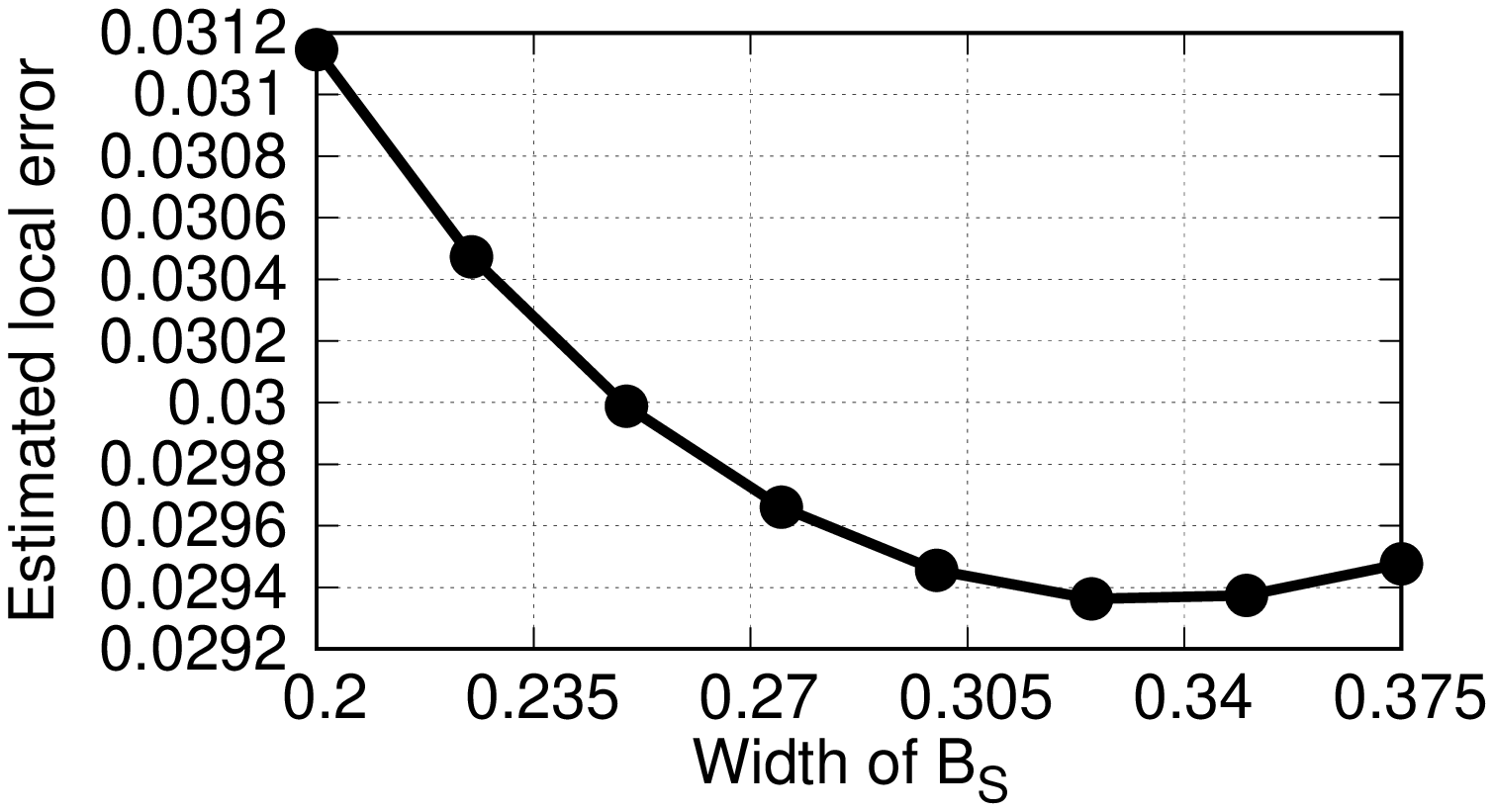} \quad
    	\includegraphics[width=3cm,angle=270]{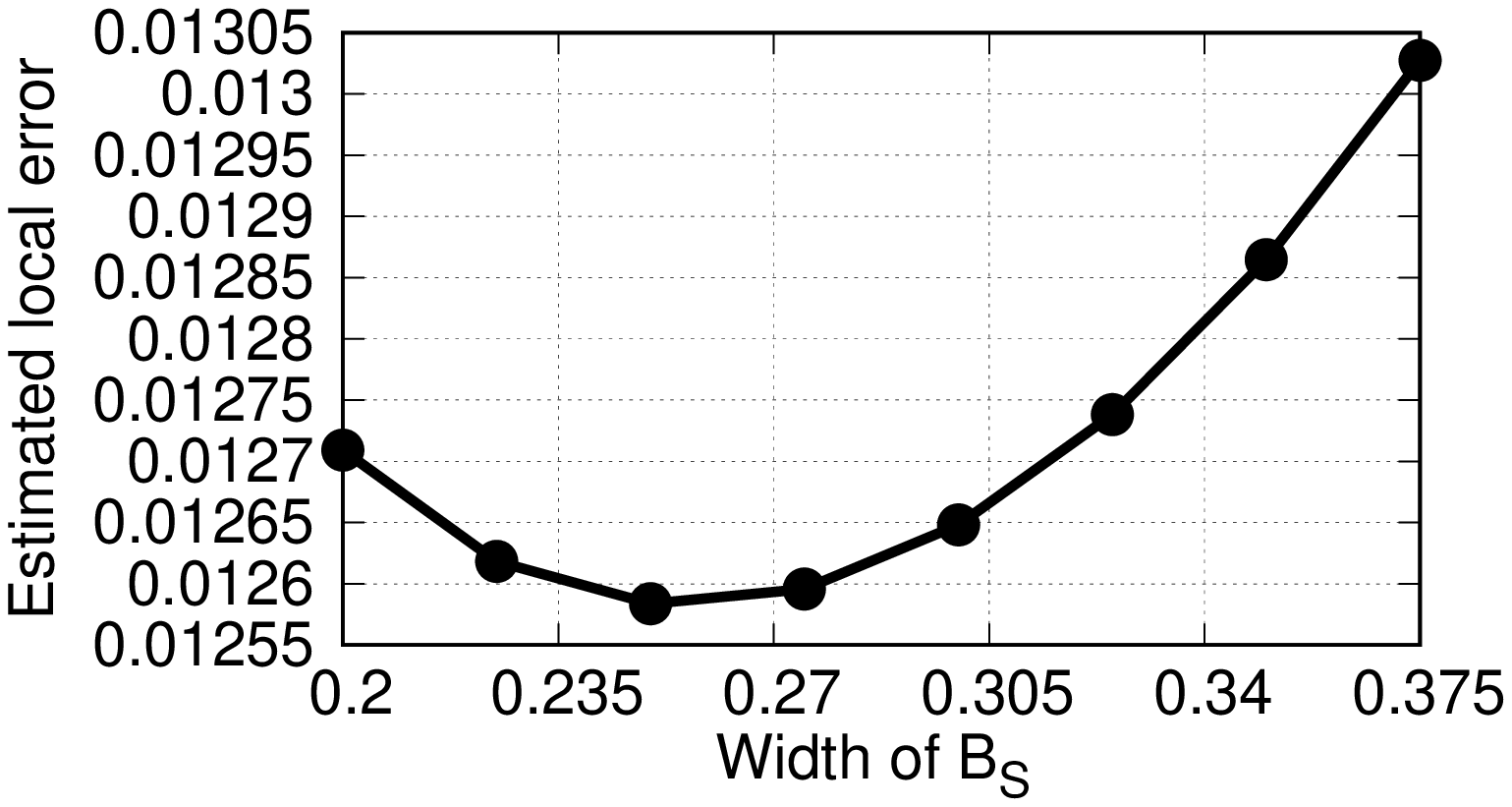}\\
	{(a) $h=1/64$} \hspace{90pt}  {(b) $h=1/128$} 
    	\caption{The relation between band width of $B_{S}$ and error estimation.}
    	\label{fig:RD-tol}
    	\end{center}
      \end{figure}

  %\noindent \textbf{メッシュサイズ$h$と誤差評価式(\ref{eq:local_est})の関係}　
  メッシュサイズ$h$と$\overline{E}_{L}$の関係をTable~\ref{table:result_R},~Fig.\ref{fig:Errest_S}-(a),(b)に示す．ここで，$B_{S}$の幅を$0.3$とした．
  
  \begin{table}[htpb]
    \caption{The relation between mesh size $h$ and error estimation}
    % \scalebox{0.65}[0.65]{ 
    \begin{center}
     \begin{tabular}{|c|cc|c|ccc|cc|}
      \hline
      \rule[-0.05cm]{0pt}{0.4cm}{} 
      $h$     & $\kappa_h$ & $C(h)$ &  $\overline{E}_{L}$ & $Err^{(1)}$     & $Err^{(2)}$    & $Err^{(3)}$     & $E_{L}$ & $\overline{E}_{G}$  \\  \hline \hline
      %\rule[-0.3cm]{0pt}{0.9cm}{} 
      1/8    & 0.057& 0.070 & {0.488}  & 0.195 & 0.337 & 0.198 & {0.131}  & 0.546      \\ 
      %\rule[-0.3cm]{0pt}{0.9cm}{} 
      1/16  &  0.030 &0.036 &{0.182}  & 0.069 & 0.123 & 0.092 & {0.063}  & 0.264      \\ 
      %\rule[-0.3cm]{0pt}{0.9cm}{} 
      1/32  &  0.015 & 0.018 &{0.070}  & 0.025 & 0.044 & 0.044 & {0.031}  & 0.129      \\ 
      %\rule[-0.3cm]{0pt}{0.9cm}{} 
      1/64  & 0.008 & 0.009 &{0.029}  & 0.009 & 0.016 & 0.022 & {0.015}  & 0.064      \\ 
      %\rule[-0.3cm]{0pt}{0.9cm}{} 
      1/128 & 0.004 &0.005 &{0.013}  & 0.003 & 0.006 & 0.011 & {0.008}  & 0.032      \\ \hline    
      \end{tabular}
      \label{table:result_R}
     \end{center}
    \end{table}
     \begin{figure}[h!]
     	\begin{center}
	\includegraphics[angle=0,width=5.5cm]{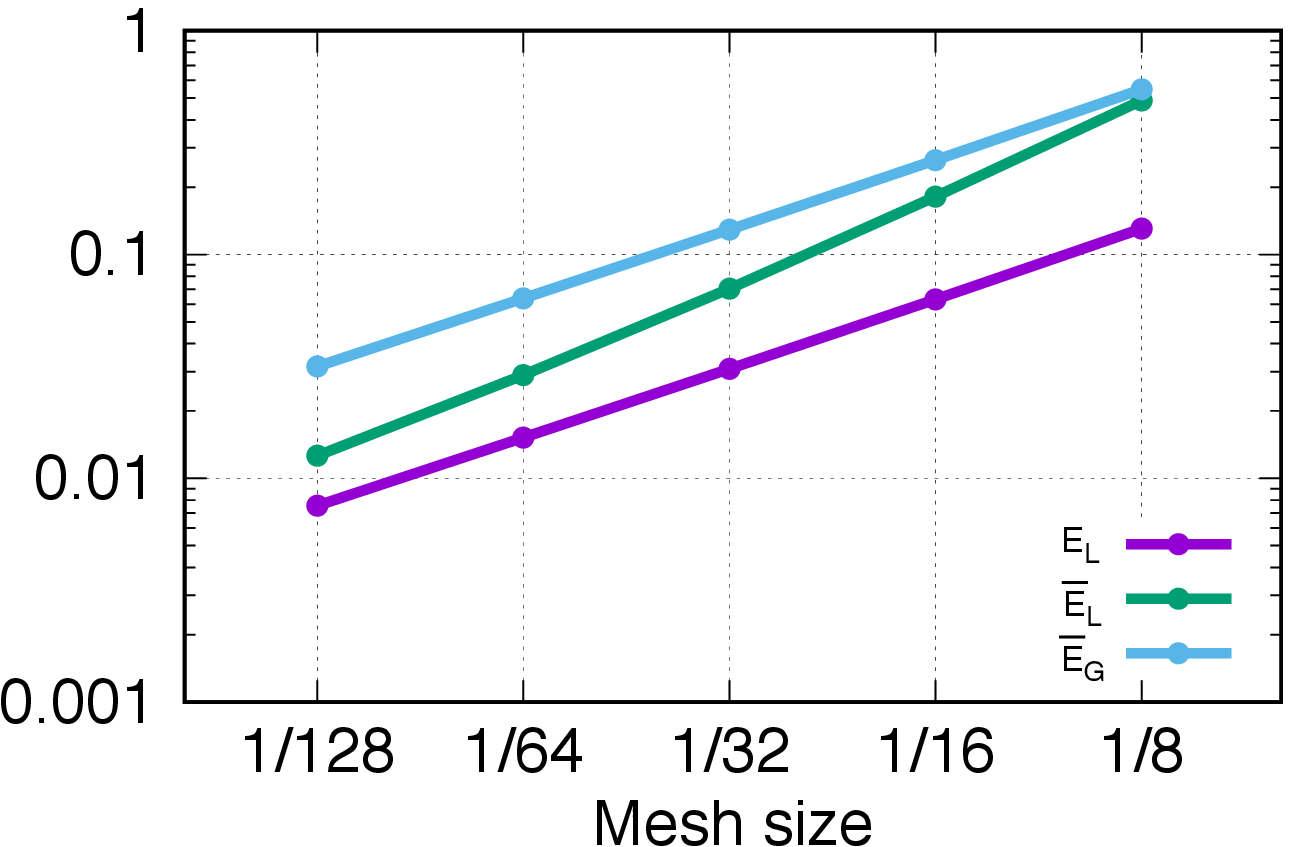} \hspace{20pt}
    	\includegraphics[angle=0,width=5.5cm]{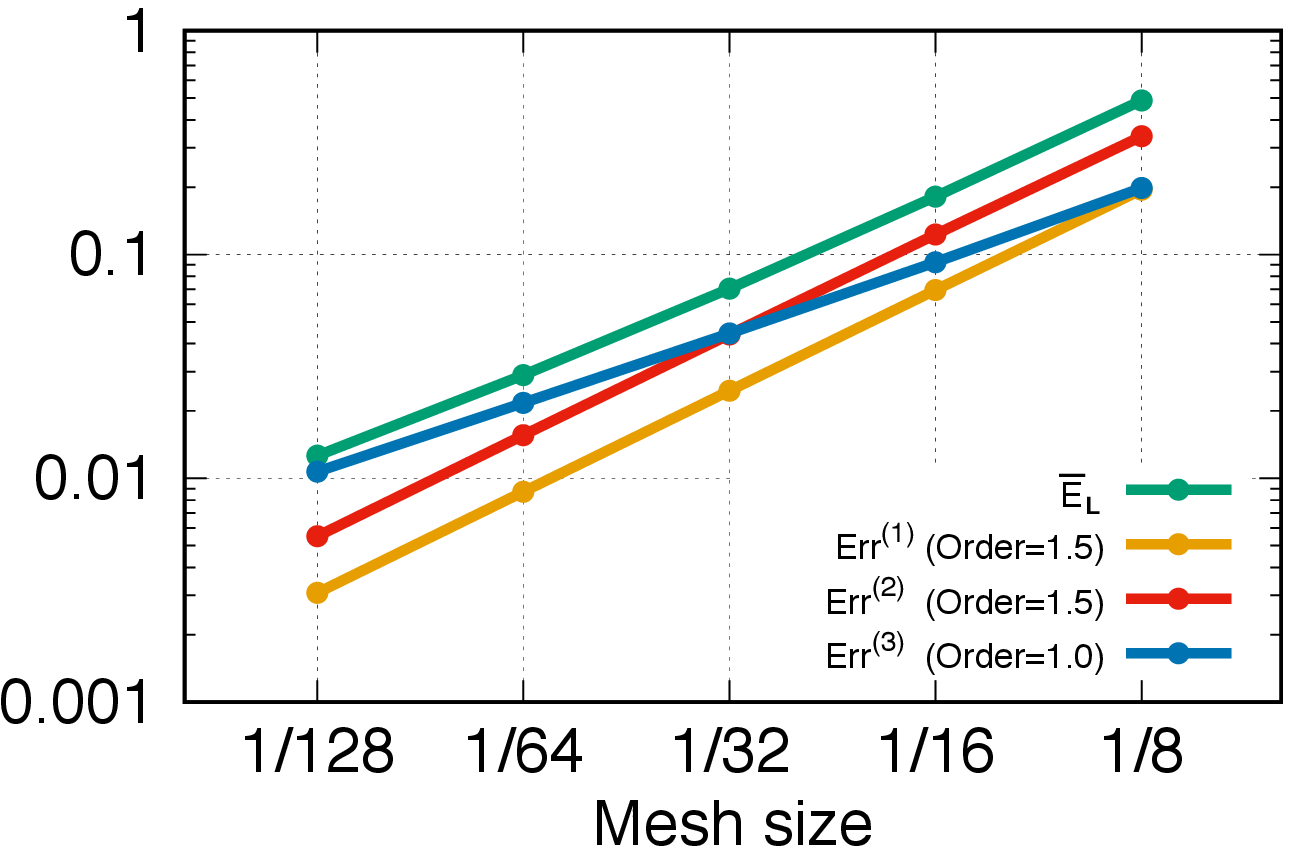} \\
	{\hspace{25pt} (a) Comparison of error estimations} \hspace{30pt}  {(b) Components of local error estimation} 
    	\caption{Numerical results}
    	\label{fig:Errest_S}
    	\end{center}
      \end{figure}
  
  Table~\ref{table:result_R},~Fig.\ref{fig:Errest_S}-(a),(b)から，$\overline{E}_{L}$は$E_{L}$の上界を与えていることがわかり，
  $h \leq 1/32$のときに誤差評価式(\ref{eq:local_est})の主要項$Err^{(3)}$が支配的である事も分かる．したがって，注意\ref{rem:order}で述べたことが数値実験よって確認できた．
%\newpage
\subsection{L字型領域}
  L字型領域上の問題(\ref{eq:experiment})の解$u$は，L字型領域の非凸な角点での特異性によって$u \notin H^2(\Omega)$となることが知られている．
  提案する誤差評価式(\ref{eq:local_est})は，解の特異性にも自然に対応できるほか，$S$と$\Omega$が境界の一部を共有する場合でも誤差評価が可能である．
 この例では，以下の条件で数値実験を行う．
 \begin{enumerate}
  \item[(a)] 領域$\Omega$: $(0,1)^2 \setminus [0.5,1]^2$，部分領域$S:=(0.375,0.625)^2 \setminus [0.5,1]^2$．
  \item[(b)] 一様メッシュのサイズ$\displaystyle h:=\frac{1}{8},~\frac{1}{16},~\frac{1}{32},~\frac{1}{64},~\frac{1}{128}$．
 \end{enumerate} 
 
 \noindent \textbf{帯状領域$B_{S}$の幅の選択} 　メッシュサイズ$h= 1/64,1/128$の場合に帯状領域$B_{S}$の幅と$\overline{E}_{L}$の関係をFig.\ref{fig:RD-tol-L}-(a),(b)に示す．
 この例の領域の設定によって，$B_{S}$の幅は最大$0.375$まで取れるので，以降の誤差評価の検討では$B_{S}$の幅を$0.375$とする．
   \begin{figure}[h!]
     	\begin{center}
	\includegraphics[width=3cm,angle=270]{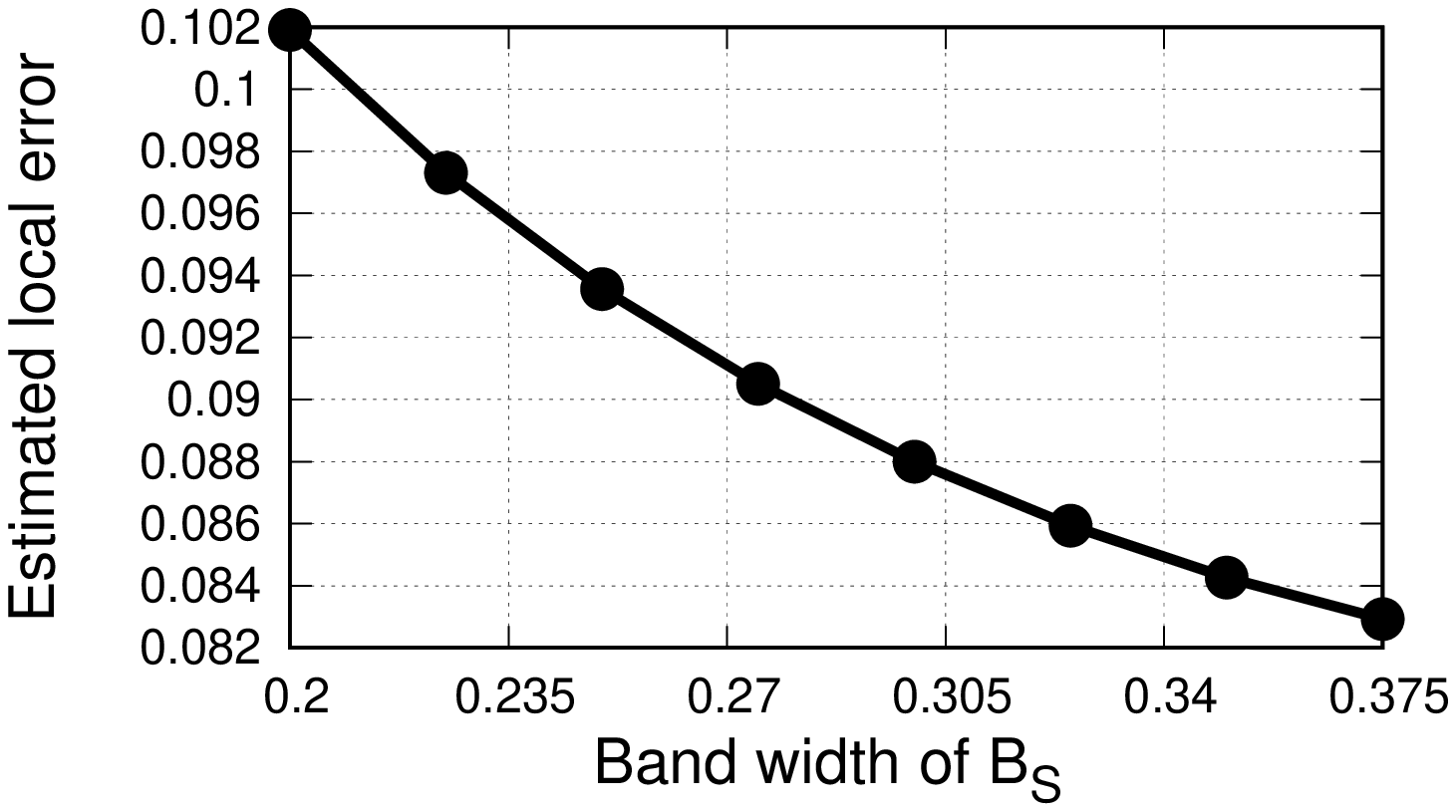} \quad
    	\includegraphics[width=3cm,angle=270]{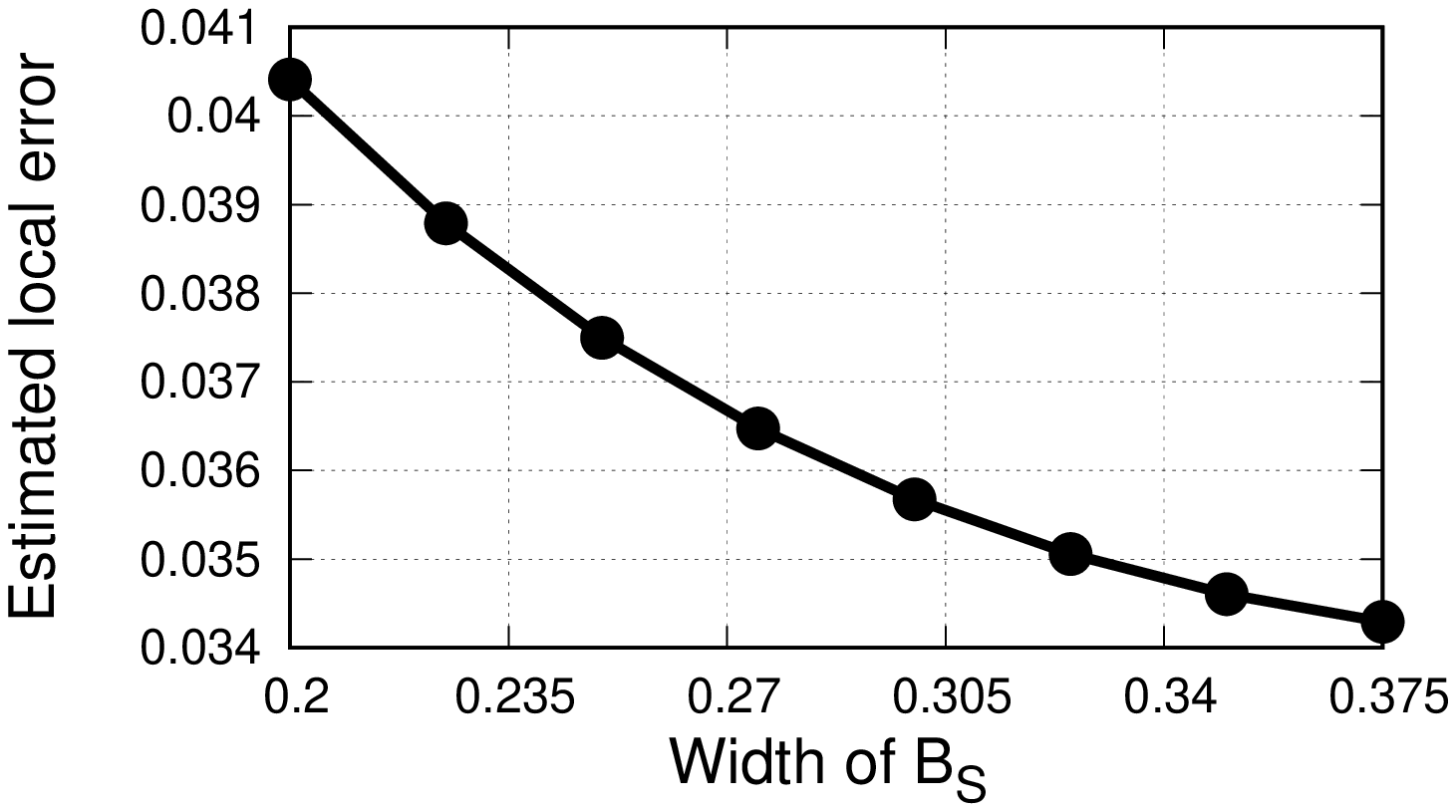}\\
	{(a) $h=1/64$} \hspace{100pt}  {(b) $h=1/128$} 
    	\caption{The relation between band width of $B_{S}$ and error estimation}
    	\label{fig:RD-tol-L}
    	\end{center}
      \end{figure}

 %\noindent \textbf{メッシュサイズ$h$と誤差評価式(\ref{eq:local_est})の関係}　
 メッシュサイズ$h$と$\overline{E}_{L}$の関係をTabel~\ref{table:result_R}, ~Fig.\ref{fig:Errest_L}-(a),(b)に示す．
  \begin{table}[h]
    \caption{The relation between mesh size $h$ and error estimation}
    \begin{center}
     \begin{tabular}{|c|cc|c|ccc|c|}
      \hline
      \rule[-0.05cm]{0pt}{0.4cm}{} 
      $h$     & $\kappa_h$ & $C(h)$ &  $\overline{E}_{L}$ & $Err^{(1)}$     & $Err^{(2)}$    & $Err^{(3)}$     & $\overline{E}_{G}$  \\  \hline \hline
      %\rule[-0.3cm]{0pt}{0.9cm}{} 
      1/8    & 0.073 & 0.083  &  {1.411} & 0.494 & 1.084 & 0.330  & 1.800      \\ 
      %\rule[-0.3cm]{0pt}{0.9cm}{} 
      1/16  & 0.046 & 0.050 & {0.542}  & 0.177 & 0.425 & 0.188   & 0.869      \\ 
      %\rule[-0.3cm]{0pt}{0.9cm}{} 
      1/32  & 0.028 & 0.029 & {0.208}  & 0.063  & 0.165 & 0.088  & 0.427      \\ 
      %\rule[-0.3cm]{0pt}{0.9cm}{} 
      1/64  & 0.017 & 0.018 & {0.083}  & 0.022 & 0.064& 0.043    & 0.211      \\ 
      %\rule[-0.3cm]{0pt}{0.9cm}{} 
      1/128 &0.011 & 0.011 & {0.034} & 0.008   & 0.025 & 0.021   & 0.105      \\ \hline    
      \end{tabular}
      \label{table:result_R}
     \end{center}
     %}
  \end{table}
     \begin{figure}[h!]
     	\begin{center}
	\includegraphics[angle=0,width=5.5cm]{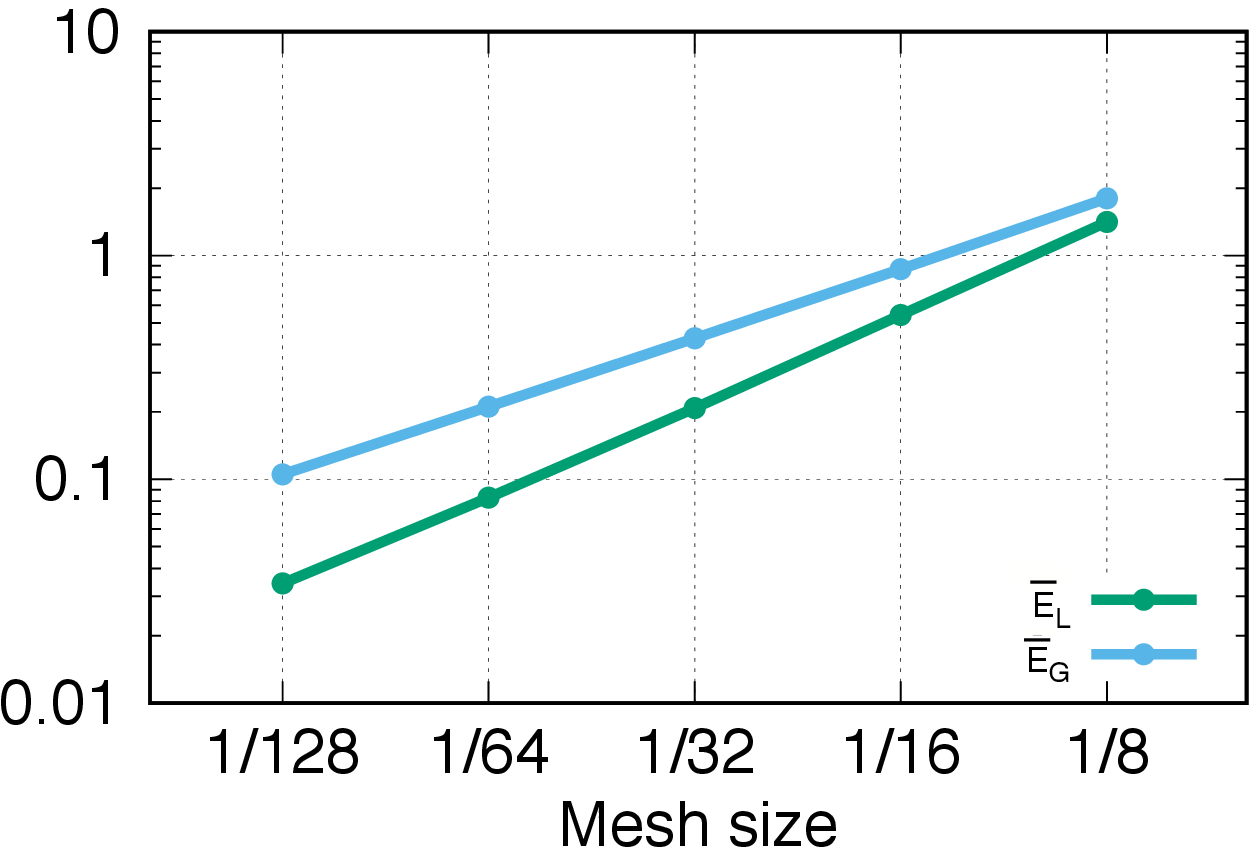} \hspace{20pt}
    	\includegraphics[angle=0,width=5.5cm]{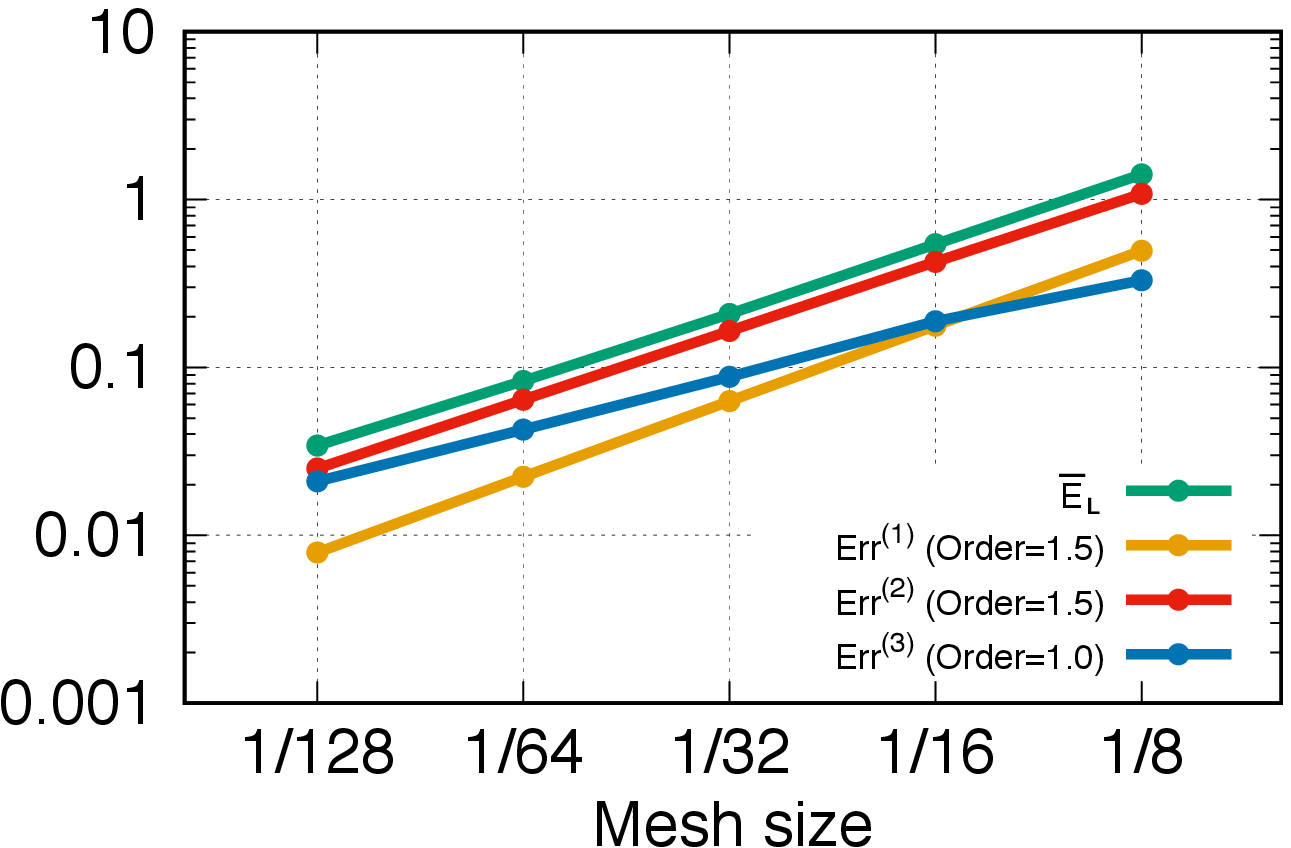} \\ 
	{\hspace{20pt} (a) Comparison of error estimations} \hspace{30pt}  {(b) Components of local error estimation} 
    	\caption{Numerical results}
    	\label{fig:Errest_L}
    	\end{center}
      \end{figure}
 
 \newpage
 \subsection{抵抗率測定法への応用}
  提案する誤差評価手法の応用例として，前述の四探針法への応用を紹介する．
  
  まず，抵抗率測定法を支配する数理モデルを準備する．
  測定試料の占める$\mathbb{R}^3$の領域を$\Omega$で表し，抵抗率測定に用いられる4本の探針$A$, $B$, $C$, $D$の位置を，それぞれ$P_A$, $P_B$, $P_C$, $P_D$で表す．
  さらに，探針$A$および$D$の接触部分を，それぞれ$\Gamma_A$, $\Gamma_D$ ($\subset \partial \Omega$)と表す．
  このとき，試料内部の電位$u$を支配する数理モデルの一つとして，以下の偏微分方程式の境界値問題が考えられる．
   \begin{eqnarray}
   -\Delta u = 0  \mbox{ in } \Omega ,~  \displaystyle \frac{\partial u}{\partial \mathbf{n}} = 0 \mbox{ on } \partial \Omega \setminus (\Gamma_A \cup \Gamma_D),~  \frac{\partial u }{\partial \mathbf{n}} = -\frac{1}{|\Gamma_A |} \mbox{ on } \Gamma_A, ~  \frac{\partial u }{\partial \mathbf{n}} = \frac{1}{|\Gamma_D |} \mbox{ on } \Gamma_D
   \label{eq:resist_model}
  \end{eqnarray}
  ここで，$|\Gamma_A|,|\Gamma_D|$は，それぞれ探針$A,D$の接触面積を表す． このとき、式(\ref{eq:v0_new})で定義される関数空間$V_0$を使用する．抵抗率測定法の問題(\ref{eq:resist_model})で関心が持たれるのは，探針$B$および$C$における電位である．
  また、有限要素法の計算では、２次元要素ではなく、四面体の分割における３次元有限要素を使用する．
  
  数値実験のために$\Omega,~P_A,~P_B,~P_C,~P_D,~\Gamma_A,~\Gamma_D$および部分領域$S$について，以下で設定する(Fig.\ref{fig:Sample}-(a)参照).
  \begin{enumerate}
  \item[(a)] 領域$\Omega$:~ $(0,8) \times (0,5) \times (0,0.2)$ \quad (直方体領域)
  \item[(b)] 探針位置:$$P_A=(3.25,2.5,0.2),~P_B=(3.75,2.5,0.2),~P_C=(4.25,2.5,0.2),~P_D=(4.75,2.5,0.2)$$
  \item[(c)] 接触部分$\Gamma_A,~\Gamma_D$:~ $\Gamma_A = B(P_A; 0.02) ,~\Gamma_D = B(P_D; 0.02)$
  \item[(d)] 関心のある部分領域$S$:~ $(3.5,4.5) \times (2.25,2.75) \times (0.0,0.2)$
  \end{enumerate}   
  ただし，$B(P;r)$は試料の表面における中心$P$，半径$r$の円である．  
  
  次に，問題(\ref{eq:resist_model})の局所事後誤差評価について検討する．
  問題(\ref{eq:resist_model})では，境界条件が区分的な定数関数によって与えられるので，定理\ref{Thm:local_est}の仮定を満足する．
  さらに，$f =0$であるため当該誤差評価式(\ref{eq:local_est})は，以下に示す簡便な式となる．
  $$\overline{E}_{L}  \leq \sqrt{  ||\nabla u_ h- \mathbf{p}_h||_\alpha^2 +2 C(h) ||\nabla \alpha||_{L^\infty(\Omega')} \cdot  ||\nabla u_h -\mathbf{p}_h||_{\Omega}^2 }$$
  
  この例では帯状領域$B_S$の幅を0.2とし，有限要素空間として2次の適合有限要素空間および1次Raviart-Thomas有限要素空間が用いられる.
  さらに$X^h$を区分的な1次多項式の関数空間とする．
  この場合，射影作用素$\pi_h$の誤差評価式\ref{eq:inter_err}については，より良い定数が利用できるが，この例では引き続き$C_0$を利用する．
  局所事後誤差評価の条件Table.\ref{table:prepair_resist}に，計算に用いたメッシュをFig.\ref{fig:Sample}-(b)に，計算結果をTable.\ref{table:reslut_resist}に示す．
 
   \begin{table}[h]
    \caption{The condition of local error estimation}
    % \scalebox{0.65}[0.65]{ 
    \begin{center}
     \begin{tabular}{|cc|ccc|}
      \hline
      \rule[-0.05cm]{0pt}{0.4cm}{} 
         Size of $A_M$ & Size of $A_C$     & $\kappa_h$    & $C_0 h$      & $C(h)$  \\  \hline \hline
         1794250 & 234571 & 0.044  & 0.053   & 0.069     \\ \hline
      \end{tabular}
      \label{table:prepair_resist}
     \end{center}
    \end{table}

  \begin{table}[h]
    \caption{Local error estimation}
    % \scalebox{0.65}[0.65]{ 
    \begin{center}
     \begin{tabular}{|c|cc|c|}
      \hline
      \rule[-0.05cm]{0pt}{0.4cm}{} 
        $\overline{E}_{L}$     & $Err^{(2)}$    & $Err^{(3)}$      & $\overline{E}_{G}$  \\  \hline \hline
       0.284 & 0.080  & 0.012   & 0.341     \\ \hline
      \end{tabular}
      \label{table:reslut_resist}
     \end{center}
    \end{table}
    ただし$A_M,~A_C$は，それぞれ混合有限要素法，適合有限要素法の全体剛性行列である．
　%ここで，Table.\ref{table:reslut_resist}中の$B,C$は，誤差評価式(\ref{eq:local_est})に現れる記号を使用している．
     \begin{figure}[h!]
     	\begin{center}
	  \hspace{40pt} \includegraphics[width=4.5cm]{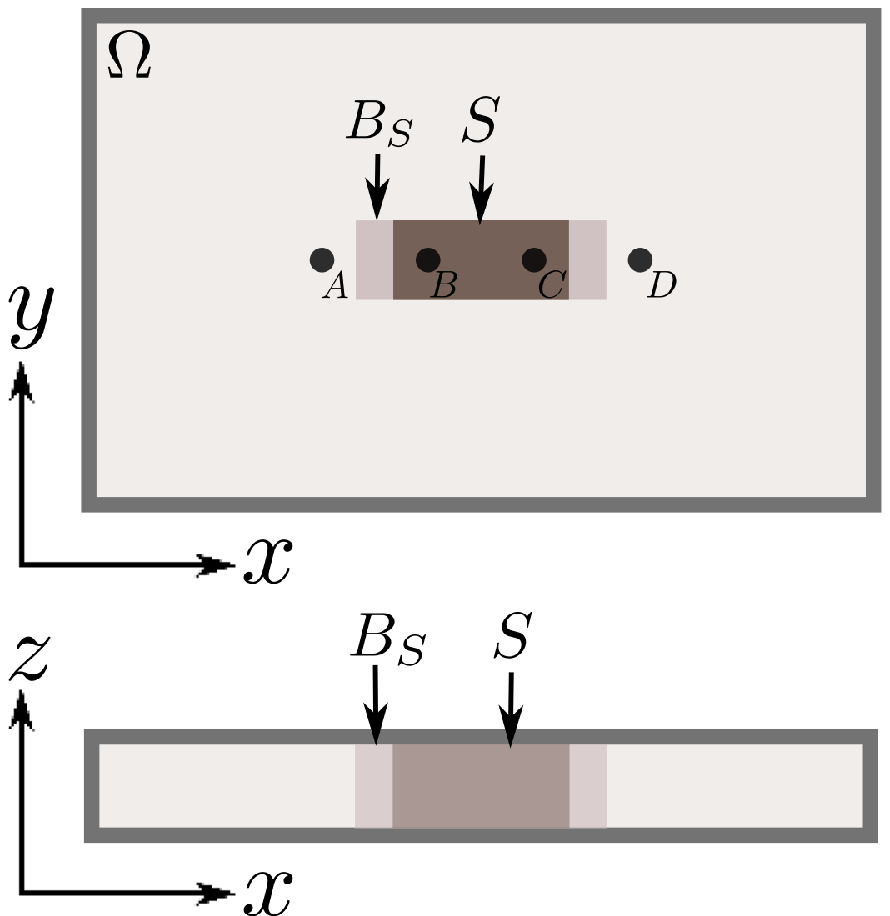} \hspace{30pt}
    	\includegraphics[width=5cm]{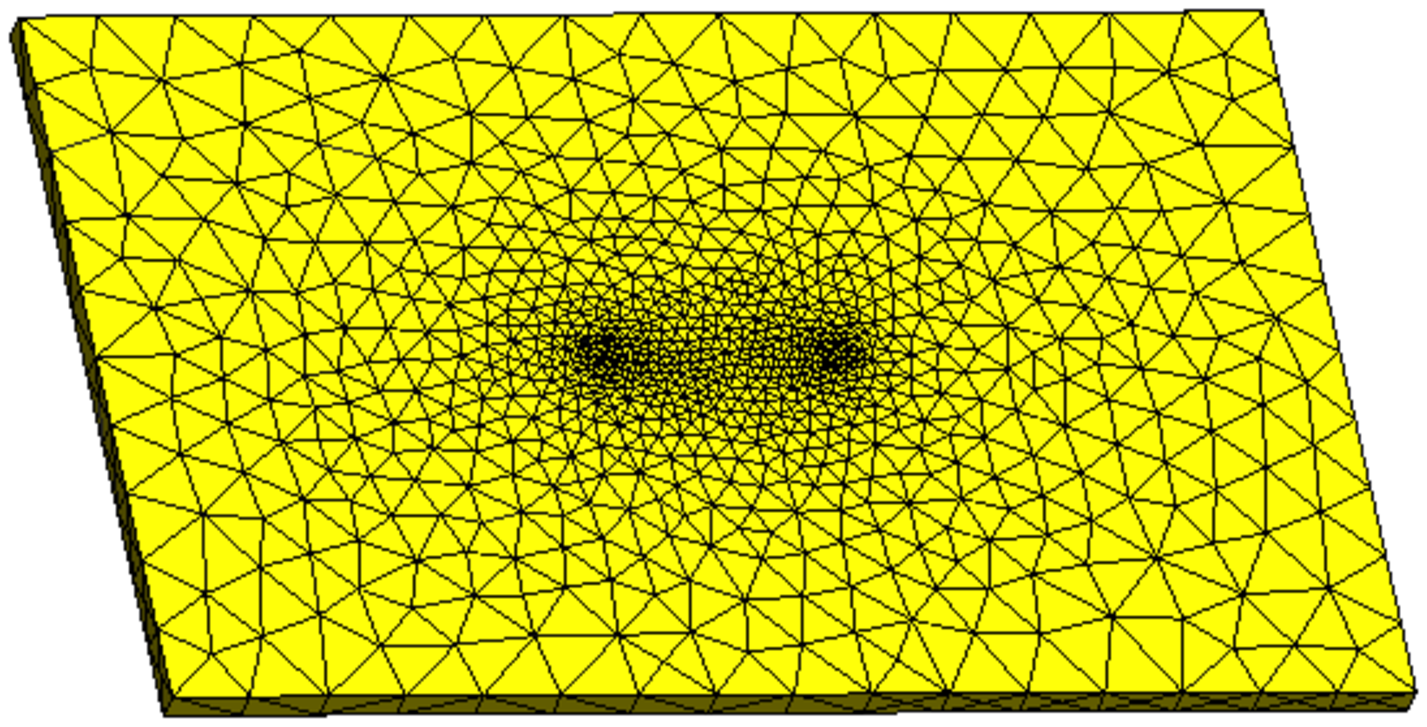} \\ 
	{(a) Problem settings of resistivity measurement} \hspace{40pt}  {(b) Tetrahedralization of $\Omega$} 
    	 \caption{Problem settings of resistivity measurement}
	 \label{fig:Sample}
    	\end{center}
      \end{figure}

\section{結論}
%半導体の抵抗率測定法を支配する偏微分方程式について近似解の局所誤差評価が要求され，
本論文では有限要素解の定量的な局所事後誤差評価について検討した．
正方形領域とL字型領域の上で定義されたPoisson方程式の境界値問題に対して数値実験を行い，提案手法の有効性について確認できた．
さらに，抵抗率測定法に関する数値例を示した．抵抗率測定法に使用する$|u(P_B)-u(P_C)|$の誤差評価については，今後の研究で検討する予定である．

\end{document}